%% file: main.tex
\documentclass[11pt, a4paper]{article}

\usepackage[T1]{fontenc}
\usepackage[utf8]{inputenc}
\usepackage{graphicx}
\usepackage{amsmath}
\usepackage{amssymb}
\usepackage{amsfonts}
\usepackage{graphicx}
\usepackage{tikz}
\usetikzlibrary{calc}
\usepackage{mathtools}
\usepackage{stmaryrd}
\usepackage{amsthm}
\usepackage{geometry}
\usepackage{graphicx}
\usepackage{enumitem}
\usepackage{hyperref}
\usepackage{authblk}
\usepackage{xparse}

\input{tikzcommands.tex}

\usepackage{setspace}

\usepackage[
  backend=biber,
  style=numeric,
  maxbibnames=99,
  sorting=nyt,
  sortcites,
  doi=false
]{biblatex}

\renewbibmacro{in:}{\ifentrytype{article}{}
    {\printtext{\bibstring{in}\intitlepunct}}}

\DeclareFieldFormat[article]{pages}{#1}

\renewbibmacro*{title}{%
  \ifboolexpr{ test {\iffieldundef{title}} }
    {}
    {\iffieldundef{doi}
       {\printfield[title]{title}}
       {\printtext[title]{\href{https://doi.org/\thefield{doi}}{\thefield{title}}}}%
     \newunit}}

\newcommand{\BG}{\mathcal{B}(G)}
\newcommand{\BlG}{\mathcal{B}_{k}(G)}
\newcommand{\BuG}{\mathcal{B}_{\geq k}(G)}
\newcommand{\multBG}{\mathcal{B}^{\kern-0.25em\raisebox{-0.1ex}{\doublecurve}}\!(G)}
\newcommand{\GClaw}{G^{\kern0.05em{\claw}}}
\newcommand{\ord}[1]{\lvert #1 \rvert}

\newcommand{\comp}[1]{\overline{#1}}
\newcommand{\type}{\text{type}}
\newcommand{\adjclaw}{\kern0.1em{\claw}}
\newcommand{\lowddagger}{\raisebox{-1.5pt}{$\scriptstyle\ddagger$}}
\newcommand{\lowclaw}{\raisebox{-1.5pt}{$\scriptstyle\claw$}}
\newcommand{\compgddagger}{\comp{G}^{\raisebox{-0.3ex}{$\scriptstyle\ddagger$}}}
\newcounter{propcounter}

\newtheorem{theorem}{Theorem}[section]

\newtheorem{lemma}[theorem]{Lemma}

\newtheorem{observation}[theorem]{Observation}

\newtheorem{conjecture}[theorem]{Conjecture}

\theoremstyle{definition}
\newtheorem{definition}[theorem]{Definition}

\title{Reconstructing a graph from its Bell colouring graph}
\author{Brian Hearn}

\affil{Department of Mathematics, London School of Economics \& Political Science, Houghton Street, London WC2A 2AE, UK \linebreak
\href{mailto:b.hearn@lse.ac.uk}{\tt b.hearn@lse.ac.uk}}

\begin{document}

\maketitle

\begin{abstract}
    The \textit{Bell colouring graph}~$\BG$ of a graph~$G$ is the graph whose vertices are the partitions of the vertex set of~$G$ into independent sets, with an edge between two partitions if and only if one can be obtained from the other by changing the part of a single vertex of~$G$.
    Given a natural number~$k$, the \textit{Bell~$k$-colouring graph}~$\BlG$ and the \textit{upper Bell~$k$-colouring graph}~$\BuG$ are the induced subgraphs of~$\BG$ consisting of all partitions with at most~$k$ parts and at least~$k$ parts, respectively.

    We determine precisely when two finite graphs have isomorphic Bell colouring graphs.
    In particular, we show that every~$n$-vertex graph~$G$ with no vertices of degree~$n-1$ can be reconstructed from its Bell colouring graph~$\mathcal{B}(G)$, and from its upper-Bell colouring graph~$\BuG$ if~$k\leq n-2$.
    We also show that every~$n$-vertex graph with maximum degree~$\Delta(G)< \frac{1}{9}n-\frac{1}{3}$ can be reconstructed from its Bell~$k$-colouring graph~$\BlG$ if~$k>\chi(G)$.
    By taking graph complements, each of these results can be restated in terms of partitions into cliques.
\end{abstract}

\input{1_introduction}
\input{2_maintheorem}
\input{3_lowertheorem}
\printbibliography
\input{A_uppertheorem}
\end{document}

%% file: 1_introduction.tex
\section{Introduction and main results}\label{section:introduction}

Let~$G=(V(G),E(G))$ be a finite graph and~$k$ a natural number.
An \emph{independent set partition} of~$G$ is a partition of~$V(G)$ into independent sets.
The \emph{Bell colouring graph~$\BG$} is the graph whose vertex set is the set of independent set partitions of~$G$, with an edge between two partitions if and only if one can be obtained from the other by changing the part of a single vertex of~$G$ (either moving a vertex into a different existing part or creating a new part of size~$1$).
Observe that~$\mathcal{B}(G)$ is a finite graph.
The \emph{Bell $k$-colouring graph~$\BlG$} is the induced subgraph of~$\BG$ whose vertex set is the set of independent set partitions of~$G$ with at most~$k$ parts.
Observe that $\BG=\mathcal{B}_{\ord{V(G)}}(G)$.
We will also consider the \textit{upper Bell $k$-colouring graph}~$\BuG$, the induced subgraph of~$\BG$ whose vertex set is the set of independent set partitions of~$G$ with at least~$k$ parts.
The term Bell $k$-colouring graph, introduced in Finbow and MacGillivray~\cite{finbow2025hamiltonicity}, comes from the \textit{$k$-Bell number} of a graph~$G$, defined as the number of partitions of~$V(G)$ into at most~$k$ independent sets.
This is, in turn, a generalisation of the classical Bell number~$B_n$ of an integer~$n$, the number of partitions of the set $\{1,\ldots,n\}$.

A (proper) \emph{$k$-colouring} of~$G$ is a map from~$V(G)$ to $\{1,2,\ldots,k\}$ such that no adjacent vertices are assigned the same value.
The \textit{chromatic number}~$\chi(G)$ of~$G$ is the smallest~$k$ required for a proper $k$-colouring of~$G$.
The \emph{$k$-recolouring graph~$\mathcal{C}_k(G)$} is the graph whose vertex set is the set of~$k$-colourings of~$G$, and whose edges link pairs of $k$-colourings which differ at exactly one vertex of~$G$.

Every $k$-colouring of~$G$ partitions the vertices of~$G$ into at most~$k$ independent sets.
Therefore, independent set partitions can be thought of as `unlabelled' colourings, making Bell $k$-colouring graphs close relatives of $k$-recolouring graphs.
The latter have been well-studied; for example, results on connectivity and Hamiltonicity can be found in \cite{dyer2006randomly,cereceda2008connectedness,choo2011gray}.
Subsequently, Bell $k$-colouring graphs have been studied in their own right, for example in Haas \cite{haas2012canonical} and Finbow and MacGillivray \cite{finbow2025hamiltonicity}, the latter showing that~$\BG$ is Hamiltonian unless~$G$ is a complete graph or a complete graph minus an edge.
Surveys on \textit{combinatorial reconfiguration} problems in general include~\cite{van2013complexity,nishimura2018introduction}.

In this paper, we are interested in determining when~$G$ is uniquely determined by its Bell colouring, Bell $k$-colouring, or upper Bell $k$-colouring graphs.
We are motivated by recent similar results for $k$-recolouring graphs~\cite{hogan2024note,berthe2025determining,asgarli2025coloring}.
Our main results are as follows.

\begin{theorem}\label{thm:main}
    Let~$G$ be a graph with no vertices of degree $\ord{V(G)}-1$.
    Then~$G$ can be reconstructed from~$\BG$.
\end{theorem}

\begin{theorem}\label{thm:mainlower}
    Let~$G$ be a graph with maximum degree $\Delta(G)<\frac{1}{9}\ord{V(G)}-\frac{1}{3}$, and let~$k>\chi(G)$.
    Then~$G$ can be reconstructed from~$\BlG$.
\end{theorem}

\begin{theorem}\label{thm:mainupper}
    Let~$G$ be a graph with no vertices of degree $\ord{V(G)}-1$, and let~$k\leq \ord{V(G)}-2$.
    Then~$G$ can be reconstructed from~$\BuG$.
\end{theorem}

We will prove each of the above theorems by giving explicit algorithms for reconstructing~$G$.
In each case, it is not necessary to know the precise value of~$k$ so long as the relevant conditions are satisfied.

Call a vertex of~$G$ \textit{universal} if it has degree $\ord{V(G)}-1$.
To see why it is impossible to detect universal vertices of~$G$ from its Bell colouring graph~$\BG$, observe that such a vertex must always belong to a part of size~$1$ in any independent set partition of~$G$.
Therefore, adding a universal vertex to~$G$ does not change its Bell colouring graph (up to isomorphism).
With this in mind, Theorems~\ref{thm:main} and~\ref{thm:mainupper} are respectively special cases of the following two results.

\begin{theorem}\label{thm:maincomplete}
    Two graphs~$G_1$ and~$G_2$ have isomorphic Bell colouring graphs if and only if the graphs obtained from~$G_1$ and~$G_2$ by removing all universal vertices are isomorphic.
\end{theorem}

\clearpage
\begin{theorem}\label{thm:mainupperhalfcomplete}
    Let~$G_1$ and~$G_2$ be graphs and let $k_i\leq \ord{V(G_i)}-2$ for $i\in \{1,2\}$.
    Then~$\mathcal{B}_{\geq k_1}(G_1)$ is isomorphic to~$\mathcal{B}_{\geq k_2}(G_2)$ if and only if the graphs obtained from~$G_1$ and~$G_2$ by removing all universal vertices are isomorphic, and either $k_i\leq \chi(G_i)$ for $i\in \{1,2\}$, or $\ord{V(G_1)}-k_1=\ord{V(G_2)}-k_2$.
\end{theorem}

In particular, if we are told~$\ord{V(G)}$, then any graph~$G$ is uniquely determined by~$\BG$, or by~$\BuG$ so long as we know that $k\leq \ord{V(G)}-2$.
Moreover, with some additional work, Theorem~\ref{thm:mainupperhalfcomplete} can be strengthened to give a complete classification of the pairs of tuples~$(G_1,k_1)$ and~$(G_2,k_2)$ such that~$\mathcal{B}_{\geq k_1}(G_1)$ is isomorphic to~$\mathcal{B}_{\geq k_2}(G_2)$.
This is Theorem~\ref{thm:mainuppercomplete} in Appendix~\ref{section:mainupper}.

Recently, the question of whether a given (unlabelled) $k$-recolouring graph~$\mathcal{C}_k(G)$ uniquely determines~$G$ has been asked and studied in \cite{asgarli2025counting,hogan2024note,berthe2025determining,asgarli2025coloring}.
Subsequently, Asgarli et al.~\cite{asgarli2025bell} showed that every tree~$T$ is uniquely determined by its (unlabelled) Bell $3$-colouring graph~$\mathcal{B}_3(T)$, and that every graph~$G$ is uniquely determined, up to any universal vertices, from its so-called (unlabelled) Bell colouring multigraph~$\multBG$, the multigraph obtained from~$\BG$ by adding multiple edges if there are multiple ways to move between two of its vertices~\cite{asgarli2025bell}.
In light of the latter result, Theorem~\ref{thm:main} shows that the multiple edges of~$\multBG$ are not necessary to determine~$G$ up to any universal vertices.
Let~$P^*$ denote the partition of~$V(G)$ into~$\ord{V(G)}$ parts, each of size~$1$.
The strategy employed in Asgarli et al.\ to reconstruct~$G$ from~$\multBG$ is to use the multiple edges in~$\multBG$ to recognise~$P^*$, from which~$G$ can be recovered by studying the structure of~$\multBG$ locally near~$P^*$.
However, without these multiple edges available to us, we do not see any way of determining~$P^*$ among the vertices of~$\BG$, and therefore require significantly more involved arguments to prove Theorem~\ref{thm:main}.
We discuss this in more detail in Section~\ref{section:main}.

Turning our attention to Theorem~\ref{thm:mainlower}, recently, Asgarli et al.~\cite{asgarli2025coloring} and Berthe et al.~\cite{berthe2025determining} independently showed that any graph is uniquely determined by its $k$-recolouring graph~$\mathcal{C}_k(G)$ so long as $k>\chi(G)$, even without knowing the precise value of~$k$, and showed that this is not true in general if $k\leq \chi(G)$.
However, the same is not true for Bell $k$-colouring graphs.
Easy examples of this fact can be found by taking any graph~$G$, and then forming a new graph~$G^+$ by adding a universal vertex.
Then, we have $\mathcal{B}_{k+1}(G^+)\cong \mathcal{B}_k(G)$ for every natural number~$k$.
(Observe that $\chi(G^+)=\chi(G)+1$.)
In an earlier version of this paper, we conjectured that all such examples are of this form, however the following small counterexamples to that conjecture were found by Rozhoň and \v{S}\'amal~\cite{XY2026}: $\mathcal{B}_2(\comp{K_3})\cong\mathcal{B}_4(K_3+K_1)\cong K_4$ and $\mathcal{B}_3(P_4)\cong \mathcal{B}_4(C_4)\cong C_4$.
Here, as usual, $K_i$,~$P_i$, and~$C_i$ refer respectively to the clique, path, and cycle on $i$ vertices,~$\comp{G}$ denotes the complement of a graph $G$, and $G_1+G_2$ denotes the disjoint union of two graphs~$G_1$ and~$G_2$.

The requirement that $k>\chi(G)$ in Theorem~\ref{thm:mainlower} is the best possible.
Indeed, if $k<\chi(G)$, then~$\BlG$ has no vertices.
If $k=\chi(G)$, Berthe et al.~\cite{berthe2025determining} gives several examples of families of graphs which are not uniquely determined by~$\mathcal{C}_{\chi(G)}(G)$, and it is straightforward to see that these graphs are not uniquely determined by~$\mathcal{B}_{\chi(G)}(G)$ either.
We restate only the simplest examples here in the context of independent set partitions: any connected bipartite graph~$G$ on at least two vertices has a unique partition into $\chi(G)=2$ independent sets, and any \textit{$k$-tree}~$T$ (i.e.\ any graph formed by starting with a $(k+1)$-clique and repeatedly adding vertices adjacent to some existing $k$-clique) has a unique partition into $\chi(T)=k+1$ independent sets.
Observe that in both examples,~$\Delta(G)$ can be made arbitrarily small relative to~$n$, showing that the requirement that $k>\chi(G)$ in Theorem~\ref{thm:mainlower} is necessary.
Two related natural questions for further research are as follows.
First, given a Bell $k$-colouring graph~$\BlG$, can we determine whether $k=\chi(G)$?
We note that the analogous question for $k$-recolouring graphs was recently answered in~\cite{asgarli2025coloring}.
Second, which graphs~$G$ are uniquely determined by their Bell $\chi(G)$-colouring graph~$\mathcal{B}_{\chi(G)}(G)$?

On the other hand, the condition that $\Delta(G)<\frac{1}{9}\ord{V(G)}-\frac{1}{3}$ in Theorem~\ref{thm:mainlower} is used to guarantee the existence of at least one partition of~$V(G)$ into~$\chi(G)$ independent sets, where each independent set has size at least~$4$.
(We do not claim that this maximum degree condition is the best possible to guarantee such a partition.)
The existence of such a partition allows us to determine~$G$ from~$\BlG$ so long as $k>\chi(G)$.
One may therefore replace the maximum degree condition in Theorem~\ref{thm:mainlower} with any condition which guarantees the existence of such a partition.
However, we believe that requiring the existence of a partition of~$V(G)$ into~$\chi(G)$ independent sets, each of size at least~$4$, is stronger than what is actually needed.
We conjecture the following.

\begin{conjecture}\label{conj:mainlower}
    Let~$G$ be a graph such that there exists a partition of $V(G)$ into~$\chi(G)$ independent sets where each such independent set has at least~$3$ vertices, and let $k>\chi(G)$.
    Then $G$ can be reconstructed from $\mathcal{B}_k(G)$.
\end{conjecture}

\begin{conjecture}\label{conj:mainlower2}
    Let~$G$ be a graph with no vertices of degree $\ord{V(G)}-1$, and let $k>\chi(G)+1$.
    Then $G$ can be reconstructed from $\mathcal{B}_k(G)$.
\end{conjecture}

We conclude this introduction by noting that every independent set partition of~$G$ is also a partition of the vertices of the complement~$\comp{G}$ into cliques.
Therefore, each theorem in this paper can be easily restated in the language of such partitions.

%% file: 2_maintheorem.tex
\section{Proof of Theorems~\ref{thm:main},~\ref{thm:mainupper},~\ref{thm:maincomplete}, and~\ref{thm:mainupperhalfcomplete}}\label{section:main}

In this section we will prove Theorems~\ref{thm:maincomplete} and~\ref{thm:mainupperhalfcomplete}, from which Theorems~\ref{thm:main} and~\ref{thm:mainupper} immediately follow.
Let~$G$ be a graph on~$n$ vertices and let~$k$ be a natural number.

We begin by defining the following terminology which will be used throughout this paper.
Let $\mathcal{B}\coloneqq \BlG$ or $\mathcal{B}\coloneqq \BuG$.
We will not distinguish between vertices of~$\mathcal{B}$ and the partitions of~$V(G)$ which they represent.

Let $P,Q\in V(\mathcal{B})$ be distinct, let $A,B\in P$ be distinct, and let $u\in A$ be given.
We say~$Q$ is obtained from~$P$ by \textit{adding}~$u$ to~$B$ if~$Q$ is obtained from~$P$ by replacing the pair $A,B\in P$ with either the pair $A\setminus\{u\},B\cup\{u\}$ (in the case where $A\neq \{u\}$) or with just $B\cup\{u\}$ (in the case where $A=\{u\}$).
We say that~$Q$ is a \textit{split neighbour} of~$P$, or the \textit{$u$-split neighbour} of~$P$, or that~$Q$ is obtained from~$P$ by \textit{splitting}~$u$, if~$Q$ is obtained from~$P$ by replacing~$A$ with $A\setminus\{u\}$ and~$\{u\}$.
(Observe that this requires that $\ord{A}\geq 2$.)
We will say that~$Q$ is obtained from~$P$ by \textit{moving}~$u$ if~$Q$ is obtained from~$P$ either by splitting~$u$ or by adding~$u$ to some part of~$P$ distinct from~$A$.
Observe that this relation is symmetric, so we may equivalently say that~$P$ is obtained from~$Q$ by moving~$u$.
Finally, if~$Q$ is obtained from~$P$ by replacing two parts $\{u\},\{v\}\in P$ with the part $\{u,v\}$, we say~$Q$ is obtained from~$P$ by \textit{merging}~$\{u\}$ and~$\{v\}$.

Given a vertex~$v\in V(G)$, we will denote the unique part of~$P$ containing~$v$ by~$P(v)$.
We will use~$N(P)$ (respectively,~$N[P]$) to denote the \textit{open neighbourhood} (respectively, \textit{closed neighbourhood}) of~$P$, the subset of $V(\mathcal{B})$ consisting of all neighbours of~$P$ (respectively, consisting of~$P$ and all neighbours of~$P$).
We will use~$\mathcal{N}(P)$ as shorthand for~$\mathcal{B}[N(P)]$, the induced subgraph of~$\mathcal{B}$ on vertex set~$N(P)$.
We will likewise use~$\mathcal{N}[P]$ as shorthand for~$\mathcal{B}[N[P]]$.

\subsection{Proof Overview}

For any graph~$G$, define~$G'$ to be the graph obtained from~$G$ by removing all of its universal vertices.
Theorem~\ref{thm:maincomplete} then claims that for any graphs~$G_1$ and~$G_2$, we have $\mathcal{B}(G_1)\cong \mathcal{B}(G_2)$ if and only if $G_1'\cong G_2'$.
The `if' direction of Theorem~\ref{thm:maincomplete} follows immediately from our observation in the introduction that adding a universal vertex to a graph~$G$ does not change its Bell colouring graph~$\BG$ (up to isomorphism).
To complete the proof of Theorem~\ref{thm:maincomplete}, it therefore remains to show that, given a Bell colouring graph~$\BG$, one can determine~$G'$.
Simultaneously, we will also show that one can determine~$G'$ from the upper Bell colouring graph~$\mathcal{B}_{\geq k}(G)$ if $k\leq n-2$.
A brief argument will then complete the proof of Theorem~\ref{thm:mainupperhalfcomplete}; this can be found in Subsection~\ref{subsect:completingupper}.

For the rest of Section~\ref{section:main}, fix a graph~$G$ on~$n$ vertices and a natural number~$k$, and let either $\mathcal{B}\coloneqq\BG$, or $\mathcal{B}\coloneqq\BuG$ with $k\leq n-1$, unless stated otherwise.

Let~$P^*$ denote the partition of the vertices of~$G$ into~$n$ parts of size~$1$.
	The strategy used in Asgarli et al.~\cite{asgarli2025bell} for determining~$G'$ from its Bell colouring multigraph~$\multBG$ has two parts.
First, they show that one can recognise~$P^*$ among the vertices of~$\multBG$ (up to automorphisms of~$\multBG$).
Next, they show that the local structure of~$\multBG$ near~$P^*$ suffices to determine~$G'$.
The existence of multiple edges in~$\multBG$ plays a crucial role in the first step.
Every neighbour of~$P^*$ is obtained from~$P^*$ by merging two singleton parts~$\{u\}$ and~$\{v\}$, hence can be obtained from~$P^*$ by moving either~$u$ or~$v$.
Therefore, in~$\multBG$, there are two edges between~$P^*$ and each of its neighbours.
This special property allows for~$P^*$ to be recognised without too much difficulty.
Indeed, Asgarli et al.\ show that any partition $P\in V(\multBG)$ with maximum degree among those with this property is equivalent to~$P^*$ up to automorphisms of~$\multBG$.

If $P^*\in V(\mathcal{B})$ can be recognised, then it is not too difficult to recover~$G'$ from~$\mathcal{B}$ (by using, for example, Lemma~\ref{lem:type1neighbours} together with Whitney's Theorem~\ref{thm:Whitney}).
However, it is not clear that there is any easy way to recognise~$P^*$ among the vertices of~$\mathcal{B}$ now that we no longer have multiple edges available to us.
Instead we will define a broader notion of a \textit{$P^*$-candidate} (Definition~\ref{defn:P^*candidate}), which can be recognised among the vertices of~$\mathcal{B}$ by direct inspection.
We will then show that there is at least one $P^*$-candidate in~$\mathcal{B}$, and give an algorithm which reconstructs~$G'$ from~$\mathcal{B}$ by using the local structure of~$\mathcal{B}$ near any $P^*$-candidate.
This will complete the proof of Theorem~\ref{thm:maincomplete}.
Some extra work in Subsection~\ref{subsect:completingupper} will then complete the proof of Theorem~\ref{thm:mainupperhalfcomplete}.

We conclude this subsection by defining this notion of a $P^*$-candidate.
Let $P\in V(\mathcal{B})$.
As usual,~$d(P)$ denotes the degree of~$P$ in~$\mathcal{B}$.
Define~$N_P$ to be the number of vertices plus the number of edges in~$\mathcal{N}(P)$.
Define~$T_P$ to be the number of triangles~$\Delta$ in~$\mathcal{N}(P)$ such that the vertices of~$\Delta$ have a common neighbour in~$\mathcal{B}$ which does not lie in~$N[P]$.

\begin{definition}\label{defn:P^*candidate}
    Call a partition $P\in V(\mathcal{B})$ a \emph{$P^*$-candidate} if all of the following properties are satisfied.
    \begin{itemize}
        \item\label{defn:prop1} \textit{Property~$1$:} Any two neighbours of~$P$ which are not adjacent to each other have exactly one common neighbour~$R$ outside of~$N[P]$, and no other neighbour of~$P$ is adjacent to~$R$.
        \item\label{defn:prop2} \textit{Property~$2$:} Let $Q_1,Q_2,Q_3\in N(P)$ be neighbours of~$P$ such that $\{Q_1,Q_2,Q_3\}$ forms a triangle in~$\mathcal{B}$, and such that $Q_1$,~$Q_2$, and~$Q_3$ have a common neighbour~$R$ outside of~$N[P]$.
        Then every other neighbour of~$P$ is adjacent to either zero or two of $Q_1$,~$Q_2$, and~$Q_3$.
        \item\label{defn:prop3} \textit{Property~$3$:} $d(P)\geq d(P')$ for every $P'\in V(\mathcal{B})$ which satisfies properties~\hyperref[defn:prop1]{$1$} and~\hyperref[defn:prop2]{$2$}.
        \item\label{defn:prop4} \textit{Property~$4$:} $N_P\geq N_{P'}$ for every $P'\in V(\mathcal{B})$ which satisfies properties~\hyperref[defn:prop1]{$1$},~\hyperref[defn:prop2]{$2$}, and~\hyperref[defn:prop3]{$3$}.
        \item\label{defn:prop5} \textit{Property~$5$:} $T_P\geq T_{P'}$ for every $P'\in V(\mathcal{B})$ which satisfies properties~\hyperref[defn:prop1]{$1$},~\hyperref[defn:prop2]{$2$},~\hyperref[defn:prop3]{$3$}, and~\hyperref[defn:prop4]{$4$}.
    \end{itemize}
\end{definition}

Observe that we can recognise the set of $P^*$-candidates among the vertices of~$\mathcal{B}$ by inspection.
We state here, without proof, that in the case where $\mathcal{B}=\mathcal{B}(G)$, it turns out that every $P\in V(\mathcal{B})$ which satisfies properties \hyperref[defn:prop1]{$1$},~\hyperref[defn:prop2]{$2$}, and~\hyperref[defn:prop3]{$3$} also satisfies properties~\hyperref[defn:prop4]{$4$} and~\hyperref[defn:prop5]{$5$}.
However, we will not need this fact in our proofs.

\subsection{The structure of partitions satisfying properties~\hyperref[defn:prop1]{$1$},~\hyperref[defn:prop2]{$2$}, and~\hyperref[defn:prop3]{$3$}}\label{subsect:mainstep1}

Our goal in this subsection is to prove Lemma~\ref{lem:props123}, which will show that~$P^*$ satisfies properties \hyperref[defn:prop1]{$1$},~\hyperref[defn:prop2]{$2$}, and~\hyperref[defn:prop3]{$3$}, and that if $P\in V(\mathcal{B})$ is a partition satisfying properties \hyperref[defn:prop1]{$1$},~\hyperref[defn:prop2]{$2$}, and~\hyperref[defn:prop3]{$3$}, then~$P$ must have a very specific form.

The following observations, stated for~$\mathcal{B}(G)$, also hold for any induced subgraph of~$\mathcal{B}(G)$.

\begin{observation}\label{obs:vopenneighbourhoodclique}
    Let $P\in V(\BG)$.
    Let~$Q_1$ and~$Q_2$ be neighbours of~$P$ in~$\BG$, and let $v\in V(G)$.
    Suppose~$Q_1$ and~$Q_2$ can each be obtained from~$P$ by moving~$v$.
    Then~$Q_2$ can be obtained from~$Q_1$ by moving~$v$.
\end{observation}

\begin{observation}\label{obs:identicalneighbours}
    Let $P\in V(\BG)$.
    Let~$Q_1$ and~$Q_2$ be neighbours of~$P$ in~$\BG$.
    If~$Q_1$ is obtained from~$P$ by moving some vertex~$v$ such that $Q_1(v)=Q_2(v)$, then we have $Q_1=Q_2$.
\end{observation}

\begin{lemma}\label{lem:type1neighbours}
    Let $P\in V(\mathcal{B})$.
    Then the following statements hold.
    \begin{enumerate}
        \item Let $Q\in N(P)$ be obtained from~$P$ by merging two parts $\{u\}$ and $\{v\}$.
        Let $Q'\in N(P)$ be obtained from $P$ by merging two parts $\{u'\}$ and $\{v'\}$.
        Then $Q$ and $Q'$ are adjacent in $\mathcal{B}$ if and only if $\{u,v\}$ and $\{u',v'\}$ are not disjoint.
        \item Let $Q_1,Q_2,Q_3\in N(P)$ be distinct partitions such that, for each $i\in \{1,2,3\}$, $Q_i$ is obtained from $P$ by merging two parts $\{u_i\}$ and $\{v_i\}$.
        Suppose $\{Q_1,Q_2,Q_3\}$ forms a triangle in $\mathcal{B}$, and that $Q_1$, $Q_2$, and $Q_3$ have a common neighbour $R$ outside of~$N[P]$.
        Then $\{u_1,v_1\}$, $\{u_2,v_2\}$, and $\{u_3,v_3\}$ are the three edges of a triangle in~$\comp{G}$.
    \end{enumerate}
\end{lemma}

\begin{proof}
    We first prove the first claim.
    Suppose $\{u,v\}$ and $\{u',v'\}$ are not disjoint.
    Without loss of generality, suppose $u=u'$.
    Then $Q'$ is obtained from $Q$ by adding $u$ to $\{v'\}$.
    Now suppose $\{u,v\}$ and $\{u',v'\}$ are disjoint.
    Then, since $\{u\}$ and $\{v\}$ are parts of $Q'$, we know that $Q'$ is the $u$-split neighbour of $Q$.
    But this is a contradiction, as the $u$-split neighbour of $Q$ is $P$.

    We now prove the second claim.
    For each $i\in \{1,2,3\}$, define $e_i\coloneqq \{u_i,v_i\}$, and observe that $e_i$ is an edge of $\comp{G}$.
    Since $\{Q_1,Q_2,Q_3\}$ forms a triangle in~$\mathcal{B}$, by the first claim we know that $e_1$, $e_2$, and $e_3$ have pairwise non-empty intersections, so are the three edges of either a triangle or a claw in $\comp{G}$.
    Suppose that $e_1$, $e_2$, and $e_3$ are the three edges of a claw in~$\comp{G}$; without loss of generality suppose that $u_1=u_2=u_3$.
    Observe that $Q_1(b)=Q_2(b)$ for all $b\in V(G)\setminus\{u_1,v_1,v_2\}$, that $Q_1(c)=Q_3(c)$ for all $c\in V(G)\setminus\{u_1,v_1,v_3\}$, and that~$R$ is obtained from~$Q_1$ by moving some vertex $a\in V(G)$.
    Therefore, we have $a\in \{u_1,v_1,v_2\}\cap \{u_1,v_1,v_3\}=\{u_1,v_1\}$ by Observation~\ref{obs:identicalneighbours}.
    However, since we know that $R\notin N[P]$, we have $a\notin \{u_1,v_1\}$ by Observation \ref{obs:vopenneighbourhoodclique}, a contradiction.
    Therefore, $e_1$, $e_2$, and $e_3$ are the three edges of a triangle in $\comp{G}$.
\end{proof}

\begin{lemma}\label{lem:P^*props1and2}
    $P^*$ satisfies properties~\hyperref[defn:prop1]{$1$} and~\hyperref[defn:prop2]{$2$}.
\end{lemma}

\begin{proof}
    Let $Q,Q'\in N(P^*)$ be non-adjacent.
    Then $Q$ is obtained from $P^*$ by merging two parts $\{u\}$ and $\{v\}$, and $Q'$ is obtained from $P^*$ by merging two parts $\{u'\}$ and $\{v'\}$.
    By Lemma~\ref{lem:type1neighbours}, we know that $\{u,v\}$ and $\{u',v'\}$ are disjoint.
    Let $R\notin N[P^*]$ be a common neighbour of~$Q$ and~$Q'$.
    Since~$Q$ is obtained from~$P^*$ by merging~$\{u\}$ and~$\{v\}$,~$R$ is not obtained from~$Q$ by moving~$u$ or~$v$ by Observation~\ref{obs:vopenneighbourhoodclique}.
    Hence $R(u)=R(v)$.
    Hence~$R$ is the partition obtained from~$Q'$ by merging~$\{u\}$ and~$\{v\}$.
    So~$Q$ and~$Q'$ have a unique common neighbour $R\notin N[P^*]$.
    Moreover, since every neighbour of~$P^*$ has exactly one part of size~$2$ and $n-2$ parts of size~$1$, any common neighbour of~$R$ and~$P^*$ is obtained from~$R$ by replacing either $\{u,v\}$ with~$\{u\}$ and~$\{v\}$ or $\{u',v'\}$ with~$\{u'\}$ and~$\{v'\}$.
    That is,~$Q$ and~$Q'$ are the only common neighbours of~$P^*$ and~$R$.
    Hence~$P^*$ satisfies property~\hyperref[defn:prop1]{$1$}.

    Let $Q_1,Q_2,Q_3\in N(P^*)$ be such that $\{Q_1,Q_2,Q_3\}$ forms a triangle in~$\mathcal{B}(G)$ with some common neighbour $R\notin N[P^*]$.
    For each $i\in \{1,2,3\}$, $Q_i$ is obtained from $P^*$ by merging two parts $\{u_i\}$ and $\{v_i\}$ to create a part $e_i\coloneqq \{u_i,v_i\}$.
    By Lemma \ref{lem:type1neighbours}, the $e_1$, $e_2$, and~$e_3$ form the three edges of a triangle in~$\comp{G}$.
    Let $Q\neq Q_1,Q_2,Q_3$ be another element of~$N(P^*)$, obtained from $P^*$ by merging two parts $\{u\}$ and $\{v\}$.
    Then $\{u,v\}$ is an edge of $\comp{G}$, so is incident with either zero or two edges of triangle with edges $e_1$, $e_2$, and $e_3$.
    Applying Lemma~\ref{lem:type1neighbours} again,~$Q$ is therefore adjacent to either zero or two of $Q_1$,~$Q_2$, and~$Q_3$.
    Hence~$P^*$ satisfies property~\hyperref[defn:prop2]{$2$}.
\end{proof}

For the rest of this subsection, let $P\in V(\mathcal{B})$ be a fixed vertex of $\mathcal{B}$.

\begin{lemma}\label{lem:nosize4}
    If~$P$ has a part of size at least~$4$, then~$P$ does not satisfy property~\hyperref[defn:prop1]{$1$}.
\end{lemma}

\begin{proof}
    Let~$A$ be a part of~$P$ of size at least~$4$, and let $a_1,a_2\in A$ be distinct.
    Let~$Q_1$ and~$Q_2$ be the partitions obtained from~$P$ by splitting~$a_1$ and~$a_2$ respectively.
    If~$Q_1$ and~$Q_2$ are adjacent, then~$Q_1$ is obtained from~$Q_2$ by splitting~$a_1$ since $\{a_1\}\in Q_1$ and $\ord{Q_2(a_1)}\geq 3$.
    But then $\{a_2\}\in Q_1$, a contradiction.
    So~$Q_1$ and~$Q_2$ are not adjacent.
    Let~$R_1$ be the $a_2$-split neighbour of~$Q_1$.
    Then~$R_1$ is the $a_1$-split neighbour of~$Q_2$.
    Let~$R_2$ be obtained from~$Q_1$ by adding~$a_2$ to~$\{a_1\}$.
    Then~$R_2$ is obtained from~$Q_2$ by adding~$a_1$ to~$\{a_2\}$.
    So both~$R_1$ and~$R_2$ are neighbours of both~$Q_1$ and~$Q_2$.
    Observe that $A\setminus \{a_1,a_2\}$ is a part of both~$R_1$ and~$R_2$.
    However, since~$A$ is a part of~$P$ of size at least~$4$, we know that $A\setminus \{a_1,a_2\}$ has size at least~$2$, and so is not a part of any neighbour of~$P$.
    So neither~$R_1$ nor~$R_2$ belong to~$N[P]$.
    So~$Q_1$ and~$Q_2$ have at least two common neighbours outside of~$N[P]$.
    So~$P$ does not satisfy property~\hyperref[defn:prop1]{$1$}.
\end{proof}

\begin{lemma}\label{lem:22edges}
    If~$P$ has two distinct parts of size~$2$ with no edges of~$G$ between them, then~$P$ does not satisfy property~\hyperref[defn:prop1]{$1$}.
\end{lemma}

\begin{proof}
    Let $\{u,v\}$ and $\{u',v'\}$ be two parts of~$P$ of size~$2$ with no edges of $G$ between them.
    We will derive a contradiction by writing down partitions $Q_1,Q_2,R_1,R_2\in V(\mathcal{B})$, which will be related to each other in $\mathcal{B}$ as illustrated in Figure~\ref{fig:22edges}.

    Let~$Q_1$ and~$Q_2$ respectively be the neighbours of~$P$ obtained from~$P$ by adding~$u'$ to $\{u,v\}$ and by adding~$u$ to $\{u',v'\}$.
    Observe that we have both $Q_1(u)=Q_1(v)$ and $Q_2(u)\neq Q_2(v)$.
    But~$Q_1$ is not obtained from~$Q_2$ by adding~$u$ to~$Q_2(v)$ nor by adding~$v$ to~$Q_2(u)$.
    Therefore,~$Q_1$ and~$Q_2$ are not adjacent in~$\mathcal{B}$.
    Let~$R_1$ be the partition obtained from~$Q_1$ by adding~$v$ to~$\{v'\}$.
    Then~$R_1$ is obtained from~$Q_2$ by adding~$v'$ to~$\{v\}$.
    Let~$R_2$ be the partition obtained from~$Q_1$ by adding~$u$ to~$\{v'\}$.
    Then~$R_2$ is obtained from~$Q_2$ by adding~$u'$ to~$\{v\}$.
    And $R_1,R_2\notin N[P]$.
    So~$Q_1$ and~$Q_2$ have two common neighbours outside of~$N[P]$.
    So~$P$ does not satisfy property~\hyperref[defn:prop1]{$1$}.
\end{proof}

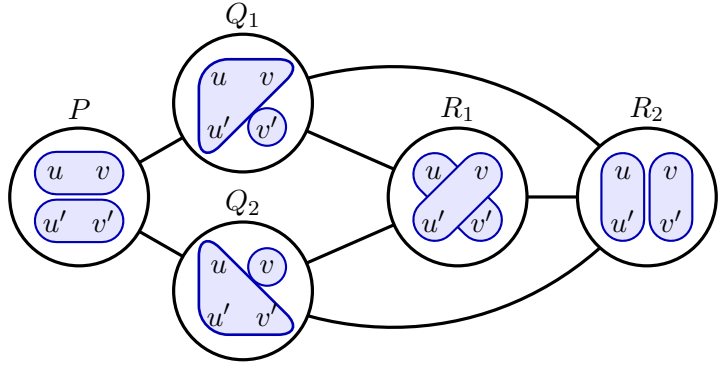
\begin{figure}[ht]
    \centering
    \begin{tikzpicture}[dot/.style={circle, fill, inner sep=2.5pt}, scale=0.142, xscale=1]
    \def\capsuleexpandfactor{1.1}
    \def\circleexpandfactor{1.7}
    \def\triangleexpandfactor{2.35}
    \def\bigcircleradius{5.5cm}
    \def\smallradius{2.7}

    \begin{scope}[]
        \coordinate (A) at (0,0);
        \coordinate (B) at (30:15cm);
        \coordinate (C) at (-30:15cm);
        \coordinate (D) at (0:30cm);
        \coordinate (E) at (0:45cm);

        \draw[very thick] (A) -- (B);
        \draw[very thick] (A) -- (C);
        \draw[very thick] (D) -- (B);
        \draw[very thick] (D) -- (C);
        \draw[very thick] (D) -- (E);
        \draw[very thick, bend right=40] (E) to (B);
        \draw[very thick, bend left=40] (E) to (C);
    \end{scope}

    \filldraw[very thick, fill = white](0,0) circle [radius=\bigcircleradius];
        \coordinate (v) at (45:\smallradius);
        \coordinate (u) at (135:\smallradius);
        \coordinate (u') at (225:\smallradius);
        \coordinate (v') at (315:\smallradius);
    \drawlabelledcapsule{u}{v}{u}{v}{\capsuleexpandfactor}
    \drawlabelledcapsule{u'}{v'}{u'}{v'}{\capsuleexpandfactor}
    \node at (0,7) {$P$};

    \begin{scope}[shift={(30:15cm)}]
        \filldraw[very thick, fill = white](0,0) circle [radius=\bigcircleradius];
        \coordinate (v) at (45:\smallradius);
        \coordinate (u) at (135:\smallradius);
        \coordinate (u') at (225:\smallradius);
        \coordinate (v') at (315:\smallradius);
        \drawlabelledroundedtriangle[\triangleexpandfactor]{u}{v}{u'}{u}{v}{u'}
        \drawlabelledcircle[\circleexpandfactor]{v'}{v'}
        \node at (0,7) {$Q_1$};
    \end{scope}

    \begin{scope}[shift={(-30:15cm)}]
        \filldraw[very thick, fill = white](0,0) circle [radius=\bigcircleradius];
        \coordinate (v) at (45:\smallradius);
        \coordinate (u) at (135:\smallradius);
        \coordinate (u') at (225:\smallradius);
        \coordinate (v') at (315:\smallradius);
        \drawlabelledroundedtriangle[\triangleexpandfactor]{u}{v'}{u'}{u}{v'}{u'}
        \drawlabelledcircle[\circleexpandfactor]{v}{v}
        \node at (0,7) {$Q_2$};
    \end{scope}

    \begin{scope}[shift={(0:30cm)}]
        \filldraw[very thick, fill = white](0,0) circle [radius=\bigcircleradius];
        \coordinate (v) at (45:\smallradius);
        \coordinate (u) at (135:\smallradius);
        \coordinate (u') at (225:\smallradius);
        \coordinate (v') at (315:\smallradius);
        \drawlabelledcapsule{u}{v'}{u}{v'}{\capsuleexpandfactor}
        \drawlabelledcapsule{u'}{v}{u'}{v}{\capsuleexpandfactor}
        \node at (0,7) {$R_1$};
    \end{scope}

    \begin{scope}[shift={(0:45cm)}]
        \filldraw[very thick, fill = white](0,0) circle [radius=\bigcircleradius];
        \coordinate (v) at (45:\smallradius);
        \coordinate (u) at (135:\smallradius);
        \coordinate (u') at (225:\smallradius);
        \coordinate (v') at (315:\smallradius);
        \drawlabelledcapsule{u}{u'}{u}{u'}{\capsuleexpandfactor}
        \drawlabelledcapsule{v}{v'}{v}{v'}{\capsuleexpandfactor}
        \node at (0,7) {$R_2$};
    \end{scope}
    \end{tikzpicture}
    \caption{The situation in the proof of Lemma~\ref{lem:22edges}.}
    \label{fig:22edges}
\end{figure}

\begin{lemma}\label{lem:21edges}
    If~$P$ has a part of size~$2$ and a part of size~$1$ with no edges of~$G$ between them, then either~$P$ does not satisfy property~\hyperref[defn:prop1]{$1$}, or $\mathcal{B}=\BuG$ with $\ord{P}=k$.
\end{lemma}

\begin{proof}
    Suppose that either $\mathcal{B}= \mathcal{B}(G)$, or $\mathcal{B}=\BuG$ with $\ord{P}>k$.
    Let $\{u,v\}$ and~$\{w\}$ be two parts of~$P$ with no edge between them.
    Let~$Q_1$ be the $u$-split neighbour of~$P$.
    Let~$Q_2$ be the partition obtained from~$P$ by adding~$w$ to $\{u,v\}$.
    Then~$Q_1$ and~$Q_2$ are neighbours of~$P$, where~$Q_1$ and~$Q_2$ are not adjacent since $u$,~$v$, and~$w$ each belong to distinct parts of~$Q_1$, but all belong to the same part of~$Q_2$.
    Let~$R$ be a common neighbour of~$Q_1$ and~$Q_2$.
    Then, since~$R$ is adjacent to~$Q_1$, we know that $u$,~$v$, and~$w$ do not all belong to the same part of~$R$.
    But since~$R$ is adjacent to~$Q_2$, at least two vertices in $\{u,v,w\}$ belong to the same part of~$R$.
    Hence~$R$ is obtained from~$Q_1$ by merging two of $\{u\}$,~$\{v\}$, and~$\{w\}$.
    So either $R=P$, or~$R$ is obtained from~$P$ by adding either~$u$ or~$v$ to~$\{w\}$.
    So every common neighbour of~$Q_1$ and~$Q_2$ lies in~$N[P]$.
    So~$P$ does not satisfy property~\hyperref[defn:prop1]{$1$}.
\end{proof}

\begin{lemma}\label{lem:31edges}
    If~$P$ has a part of size~$3$ and a part of size~$1$ with at least one non-edge of~$G$ between them, then~$P$ does not satisfy property~\hyperref[defn:prop2]{$2$}.
\end{lemma}

\begin{proof}
    Let $\{u_1,u_2,u_3\}\in P$ and $\{v\}\in P$ be parts of $P$ such that $u_1v\notin E(G)$.
    For each $i\in \{1,2,3\}$, let~$Q_i$ be the $u_i$-split neighbour of~$P$.
    Let~$R$ be the partition obtained from~$P$ by replacing $\{u_1,u_2,u_3\}$ with $\{u_1\}$,~$\{u_2\}$, and~$\{u_3\}$.
    Observe that for each distinct $i,j,k\in \{1,2,3\}$,~$Q_i$ is obtained from~$Q_j$ by adding~$u_k$ to~$\{u_j\}$.
    Moreover,~$R$ is the $u_1$-split neighbour of~$Q_2$ and~$Q_3$, and is the $u_2$-split neighbour of~$Q_1$.
    Therefore, $\{Q_1,Q_2,Q_3,R\}$ forms a clique in~$\mathcal{B}(G)$.
    Moreover, we have $Q_1,Q_2,Q_3\in N(P)$ and $R\notin N[P]$, the latter since $u_1$,~$u_2$, and~$u_3$ each belong to distinct parts of~$R$ but the same part of~$P$.
    Let~$Q'$ be the partition obtained from~$P$ by adding~$u_1$ to~$\{v\}$.
    Then we have $Q'\in N(P)$, and~$Q_1$ is the $u_1$-split neighbour of~$Q'$.
    But~$Q'$ is not adjacent to~$Q_\ell$ for $\ell\in\{2,3\}$, since, letting $m=3$ if $\ell=2$, and $m=2$ if $\ell=3$, we have $\{u_1,u_m\}\in Q_\ell$ while~$u_1$ and~$u_m$ belong to distinct parts of size~$2$ in~$Q'$.
\end{proof}

\begin{lemma}\label{lem:degreen-2}
    If~$P$ has a part $\{u,v\}$ of size~$2$ and a part~$\{w\}$ of size~$1$ with exactly one edge of~$G$ between them, then either~$w$ has degree $n-2$, or~$P$ does not satisfy property~\hyperref[defn:prop1]{$1$}, or~$P$ does not satisfy property~\hyperref[defn:prop2]{$2$}.
\end{lemma}

\begin{proof}
    Without loss of generality, let~$uw$ be an edge and~$vw$ a non-edge of~$G$.
    Let~$Q_1$ be the $u$-split neighbour of~$P$.
    Let~$Q_2$ be the partition obtained from~$P$ by adding~$v$ to~$\{w\}$.
    Suppose~$w$ does not have degree $n-2$.
    Then~$w$ has at least one non-neighbour $x\neq v$.
    If~$x$ belongs to a part of~$P$ of size at least~$4$, then by Lemma~\ref{lem:nosize4},~$P$ does not satisfy property~\hyperref[defn:prop1]{$1$}.
    If~$x$ belongs to a part of~$P$ of size~$3$, then by Lemma~\ref{lem:31edges},~$P$ does not satisfy property~\hyperref[defn:prop2]{$2$}.

    Suppose that~$x$ belongs to a part $\{x,y\}\in P$ of size~$2$, and that $vx\notin E(G)$.
    Let~$Q_3$ be the $x$-split neighbour of~$P$.
    Let~$R_1$ be the $x$-split neighbour of $Q_1$.
    Observe that $R_1$ is also the $u$-split neighbour of $Q_3$.
    Let~$R_2$ be the partition obtained from $Q_1$ by adding $x$ to $\{v\}$.
    Observe that $R_2$ is also obtained from $Q_3$ by adding $v$ to $\{x\}$.
    The situation is illustrated in Figure~\ref{fig:degreen-2fig1}.
    By construction, we have $Q_1,Q_3\in N(P)$, and both~$R_1$ and~$R_2$ are adjacent to both~$Q_1$ and~$Q_3$.
    Moreover,~$Q_1$ and~$Q_3$ are not adjacent since $\{u\},\{v\}\in Q_1$ and $\{u,v\}\in Q_3$, but~$Q_1$ is not the $u$-split neighbour of~$Q_3$.
    We have $R_1\notin N[P]$, since $\{x,y\}\in R_1$ but~$x$ and~$y$ belong to distinct parts of~$P$, each of size~$2$.
    Finally, we have $R_2\notin N[P]$ because $\{u\},\{v\}\in R_2$ and $\{u,v\}\in P$, but $R_2$ is not the $u$-split neighbour of~$P$.
    Therefore,~$P$ does not satisfy property~\hyperref[defn:prop1]{$1$}.

    Suppose that~$x$ belongs to a part $\{x,y\}\in P$ of size~$2$, and that $vx\in E(G)$.
    Let~$Q_4$ be the partition obtained from~$P$ by adding~$x$ to~$\{w\}$.
    The situation is illustrated in Figure~\ref{fig:degreen-2fig2}.
    Then $Q_2,Q_4\in N(P)$, and~$Q_2$ and~$Q_4$ are not adjacent since $\{v,w\}\in Q_2$, but $v$ and $w$ belong to distinct parts of $Q_4$, each of size~$2$.
    Suppose there is some $R\notin N[P]$ adjacent to both~$Q_2$ and~$Q_4$.
    Since $\{u,v\},\{w,x\}\in Q_4$ and $uw,vx\in E(G)$, we know that~$v$ and~$w$ belong to different parts of~$R$.
    Therefore, since $\{v,w\}\in Q_2$, we know that $R$ is obtained from $Q_2$ by moving either $v$ or $w$.
    In particular, this means that either $\{v\}\in R$ or $\{w\}\in R$, and that $R(x)=R(y)$.
    The latter fact means~$R$ is obtained from~$Q_4$ by adding~$y$ to~$\{w,x\}$ (since adding~$x$ to~$\{y\}$ gives~$P$), a contradiction as then neither~$\{v\}$ nor~$\{w\}$ are parts of~$R$.
    So no such~$R$ exists.
    Therefore,~$P$ does not satisfy property~\hyperref[defn:prop1]{$1$}.

    Suppose that~$x$ belongs to a part of~$P$ of size~$1$.
    Let~$Q_5$ be the partition obtained from~$P$ by adding~$w$ to~$\{x\}$.
    Let~$R$ be~$u$-split neighbour of~$Q_5$.
    Then~$R$ is the $w$-split neighbour of~$Q_1$, and~$R$ is obtained from~$Q_2$ by adding~$w$ to~$\{x\}$.
    But $R\notin N[P]$, since we have both $\{u\},\{v\}\in R$ and $\{u,v\}\in P$, but~$R$ is not the $u$-split neighbour of~$P$.
    Similarly,~$Q_1$ is not adjacent to~$Q_5$, since we have both $\{u\},\{v\}\in Q_1$ and $\{u,v\}\in Q_5$, but~$Q_1$ is not the $u$-split neighbour of~$Q_5$.
    So $Q_1$ and $Q_5$ are neighbours of $P$ which are not adjacent to each other in $\mathcal{B}$, but have a common neighbour outside of $N[P]$ which is adjacent to another neighbour $Q_2$ of $P$.
    Therefore,~$P$ does not satisfy property~\hyperref[defn:prop1]{$1$}.
\end{proof}

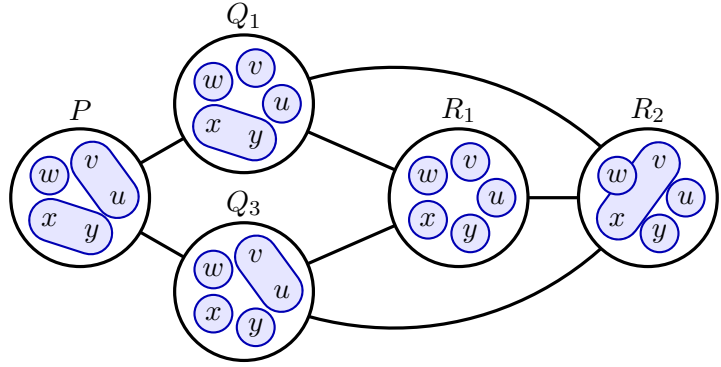
\begin{figure}[ht]
    \centering
    \begin{tikzpicture}[dot/.style={circle, fill, inner sep=2.5pt}, scale=0.142, xscale=1]
    \def\capsuleexpandfactor{1.1}
    \def\circleexpandfactor{1.7}
    \def\bigcircleradius{5.5cm}

    \begin{scope}[]
        \coordinate (A) at (0,0);
        \coordinate (B) at (30:15cm);
        \coordinate (C) at (-30:15cm);
        \coordinate (D) at (0:30cm);
        \coordinate (E) at (0:45cm);

        \draw[very thick] (A) -- (B);
        \draw[very thick] (A) -- (C);
        \draw[very thick] (D) -- (B);
        \draw[very thick] (D) -- (C);
        \draw[very thick] (D) -- (E);
        \draw[very thick, bend right=40] (E) to (B);
        \draw[very thick, bend left=40] (E) to (C);
    \end{scope}

    \filldraw[very thick, fill = white](0,0) circle [radius=\bigcircleradius];
    \coordinate (u) at (0:3);
    \coordinate (v) at (72:3);
    \coordinate (w) at (144:3);
    \coordinate (x) at (216:3);
    \coordinate (y) at (288:3);
    \drawlabelledcapsule{u}{v}{u}{v}{\capsuleexpandfactor}
    \drawlabelledcircle[\circleexpandfactor]{w}{w}
    \drawlabelledcapsule{x}{y}{x}{y}{\capsuleexpandfactor}
    \node at (0,7) {$P$};

    \begin{scope}[shift={(30:15cm)}]
        \filldraw[very thick, fill = white](0,0) circle [radius=\bigcircleradius];
        \coordinate (u) at (0:3);
        \coordinate (v) at (72:3);
        \coordinate (w) at (144:3);
        \coordinate (x) at (216:3);
        \coordinate (y) at (288:3);
        \drawlabelledcircle[\circleexpandfactor]{u}{u}
        \drawlabelledcircle[\circleexpandfactor]{v}{v}
        \drawlabelledcircle[\circleexpandfactor]{w}{w}
        \drawlabelledcapsule{x}{y}{x}{y}{\capsuleexpandfactor}
        \node at (0,7) {$Q_1$};
    \end{scope}

    \begin{scope}[shift={(-30:15cm)}]
        \filldraw[very thick, fill = white](0,0) circle [radius=\bigcircleradius];
        \coordinate (u) at (0:3);
        \coordinate (v) at (72:3);
        \coordinate (w) at (144:3);
        \coordinate (x) at (216:3);
        \coordinate (y) at (288:3);
        \drawlabelledcapsule{u}{v}{u}{v}{\capsuleexpandfactor}
        \drawlabelledcircle[\circleexpandfactor]{w}{w}
        \drawlabelledcircle[\circleexpandfactor]{x}{x}
        \drawlabelledcircle[\circleexpandfactor]{y}{y}
        \node at (0,7) {$Q_3$};
    \end{scope}

    \begin{scope}[shift={(0:30cm)}]
        \filldraw[very thick, fill = white](0,0) circle [radius=\bigcircleradius];
        \coordinate (u) at (0:3);
        \coordinate (v) at (72:3);
        \coordinate (w) at (144:3);
        \coordinate (x) at (216:3);
        \coordinate (y) at (288:3);
        \drawlabelledcircle[\circleexpandfactor]{u}{u}
        \drawlabelledcircle[\circleexpandfactor]{v}{v}
        \drawlabelledcircle[\circleexpandfactor]{w}{w}
        \drawlabelledcircle[\circleexpandfactor]{x}{x}
        \drawlabelledcircle[\circleexpandfactor]{y}{y}
        \node at (0,7) {$R_1$};
    \end{scope}

    \begin{scope}[shift={(0:45cm)}]
        \filldraw[very thick, fill = white](0,0) circle [radius=\bigcircleradius];
        \coordinate (u) at (0:3);
        \coordinate (v) at (72:3);
        \coordinate (w) at (144:3);
        \coordinate (x) at (216:3);
        \coordinate (y) at (288:3);
        \drawlabelledcircle[\circleexpandfactor]{u}{u}
        \drawlabelledcapsule{v}{x}{v}{x}{\capsuleexpandfactor}
        \drawlabelledcircle[\circleexpandfactor]{w}{w}
        \drawlabelledcircle[\circleexpandfactor]{y}{y}
        \node at (0,7) {$R_2$};
    \end{scope}

    \end{tikzpicture}
    \caption{The situation in the proof of Lemma~\ref{lem:degreen-2} if $P(x)=\{x,y\}$ and $vx\notin E(G)$.}
    \label{fig:degreen-2fig1}
\end{figure}

\begin{figure}[ht]
    \centering
    \begin{tikzpicture}[dot/.style={circle, fill, inner sep=2.5pt}, scale=0.142, xscale=1]
    \def\capsuleexpandfactor{1.1}
    \def\circleexpandfactor{1.7}
    \def\bigcircleradius{5.5cm}

    \begin{scope}[]
        \coordinate (A) at (0,0);
        \coordinate (B) at (30:15cm);
        \coordinate (C) at (-30:15cm);
        \draw[very thick] (B) -- (A) -- (C);
    \end{scope}

    \filldraw[very thick, fill = white](0,0) circle [radius=\bigcircleradius];
    \coordinate (u) at (0:3);
    \coordinate (v) at (72:3);
    \coordinate (w) at (144:3);
    \coordinate (x) at (216:3);
    \coordinate (y) at (288:3);
    \drawlabelledcapsule{u}{v}{u}{v}{\capsuleexpandfactor}
    \drawlabelledcircle[\circleexpandfactor]{w}{w}
    \drawlabelledcapsule{x}{y}{x}{y}{\capsuleexpandfactor}
    \node at (0,7) {$P$};

    \begin{scope}[shift={(30:15cm)}]
        \filldraw[very thick, fill = white](0,0) circle [radius=\bigcircleradius];
        \coordinate (u) at (0:3);
        \coordinate (v) at (72:3);
        \coordinate (w) at (144:3);
        \coordinate (x) at (216:3);
        \coordinate (y) at (288:3);
        \drawlabelledcapsule{u}{v}{u}{v}{\capsuleexpandfactor}
        \drawlabelledcapsule{w}{x}{w}{x}{\capsuleexpandfactor}
        \drawlabelledcircle[\circleexpandfactor]{y}{y}
        \node at (0,7) {$Q_4$};
    \end{scope}

    \begin{scope}[shift={(-30:15cm)}]
        \filldraw[very thick, fill = white](0,0) circle [radius=\bigcircleradius];
        \coordinate (u) at (0:3);
        \coordinate (v) at (72:3);
        \coordinate (w) at (144:3);
        \coordinate (x) at (216:3);
        \coordinate (y) at (288:3);
        \drawlabelledcapsule{v}{w}{v}{w}{\capsuleexpandfactor}
        \drawlabelledcapsule{x}{y}{x}{y}{\capsuleexpandfactor}
        \drawlabelledcircle[\circleexpandfactor]{u}{u}
        \node at (0,7) {$Q_2$};
    \end{scope}
    \end{tikzpicture}
    \caption{The situation in the proof of Lemma~\ref{lem:degreen-2} if $P(x)=\{x,y\}$ and $vx\in E(G)$.}
    \label{fig:degreen-2fig2}
\end{figure}
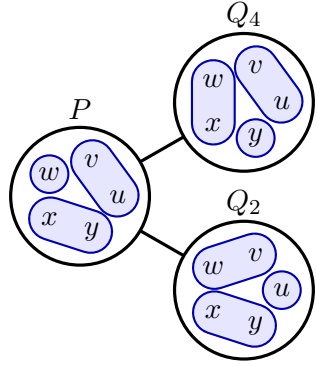

We are now ready to prove the lemma discussed at the start of this subsection.

\begin{lemma}\label{lem:props123}
    The following statements both hold.
    \begin{enumerate}
        \item $P^*$ satisfies properties \hyperref[defn:prop1]{$1$},~\hyperref[defn:prop2]{$2$}, and~\hyperref[defn:prop3]{$3$}.
        \item Suppose~$P$ satisfies properties \hyperref[defn:prop1]{$1$},~\hyperref[defn:prop2]{$2$}, and~\hyperref[defn:prop3]{$3$}.
        Then each of the following holds.
        \begin{itemize}
            \item $P$ has no parts of size greater than~$3$.
            \item Every vertex in every part of size~$3$ has degree $n-3$ in~$G$.
            \item If $A,B\in P$ are distinct parts of size~$2$, then every vertex in~$A$ is a neighbour of every vertex in~$B$.
            \item If $\{u,v\}\in P$ is a part of size~$2$ and $\{w\}\in P$ is a part of size~$1$, then at least one of the following holds.
            \begin{itemize}
                \item $w$ is a neighbour of both~$u$ and~$v$, or
                \item $w$ has degree $n-2$ in~$G$, or
                \item $\mathcal{B}=\BuG$ and $\ord{P}=k$.
            \end{itemize}
        \end{itemize}
    \end{enumerate}

\end{lemma}

\begin{proof}
    We will prove both statements simultaneously.
    Let $\Omega_{12}\subseteq V(\mathcal{B})$ denote the set of vertices of~$\mathcal{B}$ which satisfy properties~\hyperref[defn:prop1]{$1$} and~\hyperref[defn:prop2]{$2$}.
    Let $\Omega_{3}\subseteq \Omega_{12}$ denote the set of vertices of $\mathcal{B}$ which satisfy properties \hyperref[defn:prop1]{$1$},~\hyperref[defn:prop2]{$2$}, and~\hyperref[defn:prop3]{$3$}.
    By definition,~$\Omega_{3}$ is just the set of vertices in~$\Omega_{12}$ which have maximal degree in~$\mathcal{B}$.
    Observe that $\Omega_{12}\neq \emptyset$, since $P^*\in \Omega_{12}$ by Lemma~\ref{lem:P^*props1and2}.
    Let $P\in \Omega_{12}$ be given.
    It then suffices to show that $d(P)\leq d(P^*)$, and that if $d(P)=d(P^*)$, then each of the properties claimed in the statement of the lemma holds for~$P$.

    In order to relate the degrees of~$P$ and~$P^*$, we define two functions $f_{P}:P\times P\rightarrow \mathbb{N}$ and $f_{P^*}:P\times P\rightarrow \mathbb{N}$.
    Let $P'\in \{P,P^*\}$ and $A,B\in P'$ be given.
    If $A$ and $B$ are distinct, and $\ord{A}=\ord{B}=1$, and either $\mathcal{B}=\BG$ or $\mathcal{B}=\BuG$ with $\ord{P'}>k$, then set $f_{P'}(A,B)\coloneqq\frac{1}{2}$.
    Otherwise, let $f_{P'}(A,B)$ equal the number of neighbours $R\in N(P')$ which can be obtained from~$P'$ by adding a vertex from~$A$ to a part of $P'$ contained in~$B$ (i.e.\ by moving a vertex $v\in V(G)$ such that $P'(v)\subseteq A$ and $R(v)\subseteq B\cup\{v\}$).
    Observe that we can have $f_{P'}(A,B)\neq f_{P'}(B,A)$.

    For each natural number $i\in \mathbb{N}$, let $P_i\subseteq P$ denote the set of parts of~$P$ of size~$i$.
    By Lemma~\ref{lem:nosize4}, all parts of~$P$ have size at most~$3$.
    For each $P'\in \{P,P^*\}$, we therefore have
    \begin{align}\label{eqn:degreecountingsum}
        d(P')&=\sum_{j=1}^3\sum_{i=1}^j\sum_{\substack{A\in P_i\\B\in P_j}}f_{P'}(A,B),
    \end{align}
   using the fact that, given a neighbour $Q\in N(P')$, if~$Q$ is obtained from~$P'$ by merging two parts~$\{u\}$ and~$\{v\}$, then~$Q$ can be obtained from~$P'$ in exactly two ways (namely, by adding~$u$ to~$\{v\}$, or by adding~$v$ to~$\{u\}$), whereas otherwise, there is a unique vertex $w\in V(G)$ and a unique part $C\in P'$ such that~$Q$ is obtained from~$P'$ by adding~$w$ to~$C$.

    We now compare the two functions~$f_{P}$ and~$f_{P^*}$.
    \begin{itemize}
        \item For each $A\in P_1$ we have $f_P(A,A)=f_{P^*}(A,A)=0$.
        \item For each $A\in P_2$ we have $f_P(A,A)=f_{P^*}(A,A)=1$.
        \item For each $A\in P_3$ we have $f_P(A,A)=f_{P^*}(A,A)=3$.
        \item Let $A,B\in P_1$ be distinct.
            Suppose there is an edge between~$A$ and~$B$.
            Then we have $f_P(A,B)=f_{P^*}(A,B)=0$.
            Now suppose there is no edge between~$A$ and~$B$.
            Then we have $f_{P^*}(A,B)=\frac{1}{2}$, and we have $f_P(A,B)=0$ if $\mathcal{B}=\BuG$ and $\ord{P}=k$, and we have $f_P(A,B)=\frac{1}{2}$ otherwise.
            So in each case we have
            \begin{align*}
                f_P(A,B)\leq f_{P^*}(A,B).
            \end{align*}
        \item By Lemma~\ref{lem:21edges}, for each $A\in P_1$ and $B\in P_2$ we have $f_P(A,B)=f_{P^*}(A,B)=0$, and either there is one non-edge between~$A$ and~$B$, in which case $f_P(B,A)=f_{P^*}(B,A)=1$, or there are no non-edges between~$A$ and~$B$, in which case $f_P(B,A)=f_{P^*}(B,A)=0$.
        \item By Lemma~\ref{lem:31edges}, for each $A\in P_1$, $B\in P_3$, we have $f_P(A,B)=f_{P^*}(A,B)=f_P(B,A)=f_{P^*}(B,A)=0$.

    \item Let~$A, B\in P_2\cup P_3$ be distinct.
    Observe that $f_P(A,B)+f_P(B,A)$ equals the number of vertices in~$A$ with no neighbours in~$B$ plus the number of vertices in~$B$ with no neighbours in~$A$, while $f_{P^*}(A,B)+f_{P^*}(B,A)$ equals the number of non-edges from~$A$ to~$B$.
    In particular, we have
    \begin{align}\label{eqn:countingsize2and3}
        f_P(A,B)+f_P(B,A)\leq f_{P^*}(A,B)+f_{P^*}(B,A),
    \end{align}
    with equality if and only if either there are no non-edges between~$A$ and~$B$, or we have $\ord{A}=\ord{B}=2$ and there are no edges between~$A$ and~$B$.
    \end{itemize}

    Therefore, by~\eqref{eqn:degreecountingsum}, we have $d(P)\leq d(P^*)$ as claimed.
    Moreover, if $d(P)=d(P^*)$, then by the condition given to achieve equality in~\eqref{eqn:countingsize2and3}, together with Lemmas~\ref{lem:nosize4}, \ref{lem:22edges}, \ref{lem:21edges},~\ref{lem:31edges}, and~\ref{lem:degreen-2},~$P$ satisfies every condition claimed in the statement of the lemma.
\end{proof}

\subsection{The structure of $\mathcal{B}$ near vertices satisfying properties~\hyperref[defn:prop3]{$3$},~\hyperref[defn:prop4]{$4$}, and~\hyperref[defn:prop5]{$5$}}\label{subsect:mainstep2}

For the rest of this section, let $\Omega_{3}\subseteq V(\mathcal{B})$ denote the set of vertices satisfying properties~\hyperref[defn:prop1]{$1$},~\hyperref[defn:prop2]{$2$} and~\hyperref[defn:prop3]{$3$}, let $\Omega_{4}\subseteq \Omega_{3}$ denote the set of vertices of~$\Omega_{3}$ which also satisfy property~\hyperref[defn:prop4]{$4$}, and let $\Omega_{5}\subseteq \Omega_{4}$ denote the set of vertices of~$\Omega_{4}$ which also satisfy property~\hyperref[defn:prop5]{$5$}.
Observe that, by definition,~$\Omega_5$ is the set of $P^*$-candidates in~$\mathcal{B}$.
Lemma~\ref{lem:props123} tells us that $P^*\in \Omega_{3}$, and that each partition $P\in \Omega_{3}$ is forced to have a very specific form.
In this subsection, we will study the local structure of~$\mathcal{B}$ near each $P\in \Omega_{3}$, with the ultimate goal of showing that we can use the local structure of~$\mathcal{B}$ near any $P^*$-candidate (i.e.\ near any $P\in\Omega_5$) to determine the graph~$G'$.

We begin with the following classical theorem of Whitney~\cite{whitney1932}.
As usual, the \textit{line graph}~$L(H)$ of a graph~$H$ is the graph with vertex set~$E(H)$ where $e,f\in V(L(H))$ are adjacent if and only if~$e$ and~$f$ are incident in~$H$, and the \textit{claw}~$K_{1,3}$ is the graph with one vertex of degree~$3$ and three vertices of degree~$1$.
\begin{theorem}[Whitney, 1932 \cite{whitney1932}]\label{thm:Whitney}
    Let~$G_1$ and~$G_2$ be connected graphs not isomorphic to the claw~$K_{1,3}$.
    If $L(G_1)\cong L(G_2)$, then $G_1\cong G_2$.
\end{theorem}

Observe that the disjoint union $G_1+G_2$ of two graphs~$G_1$ and~$G_2$ satisfies $L(G_1+G_2)=L(G_1)+L(G_2)$, that $L(K_{1,3})\cong L(K_3)$, and that~$L(K_1)$ is the graph on no vertices.
Therefore, Whitney's Theorem says that we can `almost' determine any graph~$H$ from its line graph~$L(H)$, except that we cannot distinguish components of~$H$ isomorphic to the claw~$K_{1,3}$ from those isomorphic to the triangle~$K_3$, and that we may miss isolated vertices of~$H$.
We therefore obtain the following, which we will use at the end of this subsection.

\begin{observation}\label{obs:lineinverse}
    Let~$H$ be a graph.
    Let~$H^{\ddagger}$ be the graph obtained from~$H$ by removing all of its isolated vertices, and by replacing all of its triangle components with claw components.
    Then~$H^{\ddagger}$ can be determined from~$L(H)$.
\end{observation}

An outline of what we will show in this subsection is as follows.
First, we will show that the open neighbourhood~$\mathcal{N}(P^*)$ is isomorphic to the line graph~$L(\comp{G})$, and that for each $P\in \Omega_{3}$, the open neighbourhood~$\mathcal{N}(P)$ is isomorphic to a subgraph of~$L(\comp{G})$.
Recall that, for each $P\in V(\mathcal{B})$, we use~$N_P$ to denote the number of vertices plus the number of edges in~$\mathcal{N}(P)$, and that $\Omega_{4}\subseteq \Omega_{3}$ is the set of $P\in \Omega_{3}$ such that $N_P\geq N_{P'}$ for all $P'\in \Omega_{3}$.
Therefore we have $P^*\in \Omega_{4}$, and for each $P\in \Omega_4$ we have $\mathcal{N}(P)\cong L(\comp{G})$.
Hence, by Whitney's Theorem~\ref{thm:Whitney}, we can almost determine~$G'$ from any $P\in \Omega_{4}$, except that we do not yet have a way to distinguish between triangles and claws in~$\comp{G}$.
Fix any $P\in \Omega_{4}$.
Then each triangle in~$L(\comp{G})\cong \mathcal{N}(P)$ corresponds to either a triangle or a claw in~$\comp{G}$.
We will show the following.
\begin{enumerate}
    \item For each triangle~$\Delta$ in~$\mathcal{N}(P)$, if~$\Delta$ arose from a claw in~$\comp{G}$, then every common neighbour of the three vertices in~$\Delta$ lies in~$N[P]$.
    \item For each triangle~$\Delta^*$ in~$\mathcal{N}(P^*)$, if~$\Delta^*$ arose from a triangle in~$\comp{G}$, then the three vertices in~$\Delta^*$ have a common neighbour $R\notin N[P^*]$.
\end{enumerate}

Recall that, for each $P'\in V(\mathcal{B})$, we use~$T_{P'}$ to denote the number of triangles in~$\mathcal{N}(P')$ which have a common neighbour $R\notin N[P']$, and that $\Omega_5\subseteq \Omega_4$ is the set of vertices $P\in \Omega_4$ such that $T_P\geq T_{P'}$ for every $P'\in \Omega_{4}$.
Therefore, for each $P\in \Omega_5$, we know that each triangle~$\Delta$ in~$\mathcal{N}(P)$ arose from a triangle in~$\comp{G}$ if and only if the vertices of~$\Delta$ have a common neighbour outside of~$N[P]$, and arose from a claw otherwise.
Then we can reconstruct~$G'$ from the local structure of~$\mathcal{B}$ near any $P\in \Omega_5$, as required.

We begin by defining the following useful function, which associates each neighbour of a vertex $P\in \Omega_{3}$ with a non-edge $uv\in E(\comp{G})$.

\begin{definition}\label{defn:neighbourtypes}
    Let $P\in \Omega_{3}$.
    Define the map $\psi_P:N(P)\rightarrow E(\comp{G})$ for each $Q\in N(P)$ as follows.
    \begin{itemize}
        \item If~$Q$ is obtained from~$P$ by merging two parts~$\{u\}$ and~$\{v\}$, let $\psi_P(Q)\coloneqq\{u,v\}$, and call~$Q$ a \textit{type~$1$} neighbour of~$P$.
        \item If~$Q$ is obtained from~$P$ by replacing a part $\{u,v\}$ with two parts~$\{u\}$ and~$\{v\}$, let $\psi_P(Q)\coloneqq\{u,v\}$, and call~$Q$ a \textit{type~$2$} neighbour of~$P$.
        \item If~$Q$ is obtained from~$P$ by replacing two parts $\{u,v\}$ and~$\{w\}$ with two parts~$\{u\}$ and $\{v,w\}$, let~$\psi_P(Q)\coloneqq\{v,w\}$, and call~$Q$ a \textit{type~$3$} neighbour of~$P$ if~$w$ has degree $n-2$ and call~$Q$ a \textit{type~$4$} neighbour of~$P$ otherwise.
        \item If~$Q$ is obtained from~$P$ by replacing a part $\{u,v,w\}$ with two parts~$\{u\}$ and $\{v,w\}$ (where~$u,v$, and~$w$ all have degree $n-3$ by Definition~\ref{defn:P^*candidate}), let $\psi_P(Q)\coloneqq\{v,w\}$, and call~$Q$ a \textit{type~$5$} neighbour of~$P$.
    \end{itemize}
\end{definition}

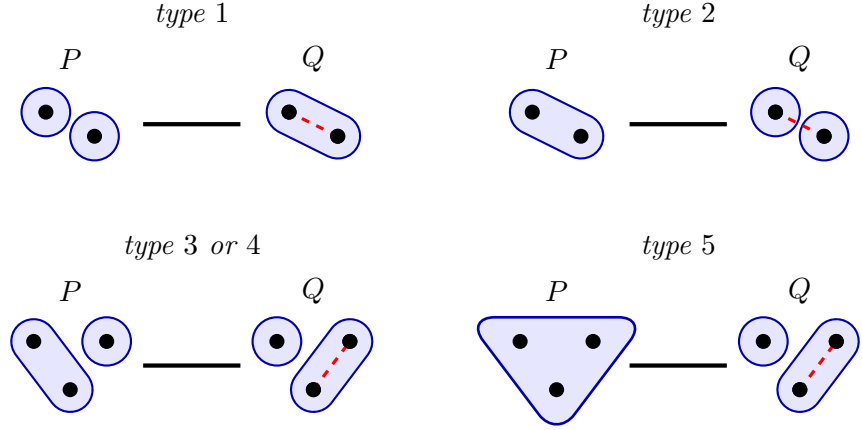
\begin{figure}[ht]
    \centering
    \begin{tikzpicture}[dot/.style={circle, fill, inner sep=2pt}, scale=0.25, xscale=1]
    \def\capsuleexpandfactor{1}

    \begin{scope}
        \node[] at (-1,3.2)  {$P$};
        \node[] at (9,3.2)  {$Q$};

        \node[] at (4,5)  {\textit{type $1$}};
        \coordinate (A) at (0,0);
        \coordinate (B) at (-2,1);
        \drawcircle[1]{A};
        \drawcircle[1]{B};

        \coordinate (A') at (10,0);
        \coordinate (B') at (8,1);
        \drawcapsule{A'}{B'}{\capsuleexpandfactor}

        \draw[ultra thick] (2,0.5) -- (6,0.5);

        \draw[dashed, red, very thick] (A') -- (B');
        \node[circle, fill, inner sep=2pt] at (A') {};
        \node[circle, fill, inner sep=2pt] at (B') {};
    \end{scope}

    \begin{scope}[xshift=20cm]
        \node[] at (-1,3.2)  {$P$};
        \node[] at (9,3.2)  {$Q$};

        \node[] at (4,5)  {\textit{type $2$}};
        \coordinate (A) at (0,0);
        \coordinate (B) at (-2,1);
        \drawcapsule{A}{B}{\capsuleexpandfactor}

        \coordinate (A') at (10,0);
        \coordinate (B') at (8,1);
        \drawcircle[1]{A'};
        \drawcircle[1]{B'};

        \draw[ultra thick] (2,0.5) -- (6,0.5);

        \draw[dashed, red, very thick] (A') -- (B');
        \node[circle, fill, inner sep=2pt] at (A') {};
        \node[circle, fill, inner sep=2pt] at (B') {};
    \end{scope}

    \begin{scope}[yshift=-10cm]
        \node[] at (-1,3.6)  {$P$};
        \node[] at (9,3.6)  {$Q$};

        \node[] at (4,5.4)  {\textit{type $3$ or $4$}};
        \coordinate (A) at (-1,-0.5);
        \coordinate (B) at (-2.5,1.5);
        \coordinate (C) at (0.5,1.5);
        \drawcapsule{A}{B}{\capsuleexpandfactor}
        \drawcircle[1]{C};

        \coordinate (A') at (9,-0.5);
        \coordinate (B') at (7.5,1.5);
        \coordinate (C') at (10.5,1.5);
        \drawcapsule{A'}{C'}{\capsuleexpandfactor}
        \drawcircle[1]{B'};

        \draw[ultra thick] (2,0.5) -- (6,0.5);

        \draw[dashed, red, very thick] (A') -- (C');
        \node[circle, fill, inner sep=2pt] at (A') {};
        \node[circle, fill, inner sep=2pt] at (C') {};
    \end{scope}

    \begin{scope}[xshift=20cm, yshift = -10cm]
        \node[] at (-1,3.6)  {$P$};
        \node[] at (9,3.6)  {$Q$};

        \node[] at (4,5.4)  {\textit{type $5$}};
        \coordinate (A) at (-1,-0.5);
        \coordinate (B) at (-2.5,1.5);
        \coordinate (C) at (0.5,1.5);
        \drawroundedtriangle[2.5]{A}{B}{C}

        \coordinate (A') at (9,-0.5);
        \coordinate (B') at (7.5,1.5);
        \coordinate (C') at (10.5,1.5);
        \drawcapsule{A'}{C'}{\capsuleexpandfactor}
        \drawcircle[1]{B'};

        \draw[ultra thick] (2,0.5) -- (6,0.5);
        \draw[dashed, red, very thick] (A') -- (C');
        \node[circle, fill, inner sep=2pt] at (A') {};
        \node[circle, fill, inner sep=2pt] at (C') {};
    \end{scope}
    \end{tikzpicture}
    \caption{The five types of neighbours~$Q$ of a $P^*$-candidate~$P$.
    For each type, the dashed line represents the non-edge $\psi_P(Q)\in E(\comp{G})$.}
    \label{fig:neighbourtypes}
\end{figure}

For each $Q\in N(P)$, let $\type(Q)\in\{1,2,3,4,5\}$ represent the type of~$Q$.
These five types of neighbour are illustrated in Figure~\ref{fig:neighbourtypes}.
For ease of notation, we will define $\psi\coloneqq \psi_P$ when the choice of~$P$ is clear from context.

\begin{lemma}\label{lem:psibij}
    Let $P\in \Omega_{3}$.
    Then~$\psi_P$ is well-defined and injective.
    Moreover,~$\psi_{P^*}$ is bijective.
\end{lemma}

\begin{proof}
    By Lemma~\ref{lem:props123}, every non-edge $uv\in E(\comp{G})$ either lies within a part of~$P$ of size~$3$ or~$2$, or between a part of~$P$ of size~$2$ and a part of~$P$ of size~$1$, or between two parts of~$P$ of size~$1$.
    If $P=P^*$, then in addition, for each non-edge of~$G$, there is exactly one $Q\in N(P)$ such that $\psi_{P^*}(Q)=\{u,v\}$.
\end{proof}

Given two neighbours~$Q$ and~$Q'$ of a vertex $P\in \Omega_3$, the following lemma records the cases where it is possible that~$\psi(Q)$ and~$\psi(Q')$ are incident as edges of~$\comp{G}$.

\begin{lemma}\label{lem:pairsoftypes}
    Let $P\in \Omega_{3}$, and let $Q,Q'\in N(P)$ be distinct neighbours of~$P$, where $\type(Q)\leq \type(Q')$.
    If~$\psi(Q)$ and~$\psi(Q')$ are incident (as edges of~$\comp{G}$), then either $\type(Q)=\type(Q')\in\{1,3,4,5\}$, or $\type(Q)\in\{1,2\}$ and $\type(Q')\in \{3,4\}$.
\end{lemma}

\begin{proof}
    Let $v\in\psi(Q)\cap \psi(Q')$.
    If $\type(Q)=1$ or $\type(Q')=1$ then we have $\ord{P(v)}=1$.
    If $\type(Q)=2$ or $\type(Q')=2$ then we have $\ord{P(v)}=2$.
    If $\type(Q)\in \{3,4\}$ or $\type(Q')\in \{3,4\}$ then we have $\ord{P(v)}\in \{1,2\}$.
    Finally, if $\type(Q)=5$ or $\type(Q')=5$ then we have $\ord{P(v)}=3$.
    Therefore, we know that either $\type(Q)=\type(Q')$, or $\type(Q)\in\{1,2\}$ and $\type(Q')\in \{3,4\}$.
    Moreover, if $\type(Q)=\type(Q')=2$, then we have $\psi(Q),\psi(Q')\in P$, so~$\psi(Q)$ and~$\psi(Q')$ are disjoint, a contradiction.
\end{proof}

\begin{lemma}\label{lem:linegraph}
     Let $P\in \Omega_{3}$, and let $Q,Q'\in N(P)$ be distinct neighbours of~$P$.
     If $QQ'\in E(\mathcal{B})$, then~$\psi(Q)$ and~$\psi(Q')$ are incident as edges of~$\comp{G}$.
\end{lemma}

\begin{proof}
    Suppose~$\psi(Q)$ and~$\psi(Q')$ are not incident.
    Then, by definition of~$\psi$, there are at least four vertices~$v$ such that $Q(v)\neq Q'(v)$.
    Suppose $QQ'\in E(\mathcal{B})$.
    Then~$Q'$ must be obtained from~$Q$ by moving a vertex such that $\ord{Q(v)\cup Q'(v)}\geq 4$, i.e.\ such that $\ord{Q(v)}+\ord{Q'(v)}\geq 5$.
    But every part of~$P$ has size at most~$3$, and there is at least one edge between each part of size~$3$ and every other part of~$P$.
    So every part of~$Q$ has size at most~$3$.
    Therefore, without loss of generality $\ord{Q(v)}=3$ and $\ord{Q'(v)}\in \{2,3\}$.
    Hence~$v$ has degree at most $n-4$ and belongs to a part of size~$3$ in~$Q$.
    But by Lemma~\ref{lem:props123}, it is not possible that~$Q$ is obtained from~$P$ by adding a vertex to a part of size~$2$, since any vertex not adjacent to both vertices in a part of~$P$ of size~$2$ belongs to a part of~$P$ of size~$1$, and if such a vertex exists, we have $\mathcal{B}=\BuG$ with $\ord{P}=k$, meaning such a neighbour would have $k-1$ parts, a contradiction.
    Therefore, since every part of~$P$ has size at most~$3$, every part of size~$3$ in~$Q$ is a part of size~$3$ in~$P$.
    Therefore,~$v$ belongs to a part of size~$3$ in~$P$.
    Therefore~$v$ has degree~$n-3$ by Lemma~\ref{lem:props123}, a contradiction.
    So $QQ'\notin E(\mathcal{B})$.
\end{proof}

\begin{lemma}\label{lem:linegraph2}
    The following statements hold.
    \begin{enumerate}
        \item We have $P^*\in \Omega_4$.
        \item Let $P\in \Omega_4$.
        Let $Q_1,Q_2\in N(P)$ be distinct neighbours of~$P$.
        Then $Q_1Q_2\in E(\mathcal{B})$ if and only if~$\psi(Q_1)$ and~$\psi(Q_2)$ are incident as edges of~$\comp{G}$.
        Moreover,~$\psi_P$ is bijective, and~$\mathcal{N}(P)$ is isomorphic to the line graph~$L(\comp{G})$.
    \end{enumerate}
\end{lemma}

\begin{proof}
    We will prove both statements simultaneously.
    By Lemma~\ref{lem:type1neighbours}, we know that two neighbours $Q,Q'\in N(P^*)$ are adjacent in $\mathcal{B}$ if and only if $\psi_{P^*}(Q)$ and $\psi_{P^*}(Q')$ are incident as edges of $\comp{G}$.
    Since~$\psi_{P^*}$ is bijective by Lemma~\ref{lem:psibij}, it follows that $\mathcal{N}(P^*)\cong L(\comp{G})$ by definition of a line graph.

    Let $P\in \Omega_3$ be given, and let $Q_1,Q_2\in N(P)$ be distinct neighbours of~$P$.
    By Lemma~\ref{lem:linegraph} we know that if $Q_1Q_2\in E(\mathcal{B})$, then~$\psi(Q_1)$ and~$\psi(Q_2)$ are incident.
    Therefore, since~$\psi_P$ is injective by Lemma~\ref{lem:psibij}, we know that~$\mathcal{N}(P)$ is isomorphic to a subgraph of the line graph~$L(\comp{G})$.
    It follows that $N_P\leq N_{P^*}$.
    Therefore, we have $P^*\in \Omega_4$.
    Moreover, if~$\psi_P$ is not surjective, or if~$\psi(Q_1)$ and~$\psi(Q_2)$ are incident and $Q_1Q_2\notin E(\mathcal{B})$, then~$\mathcal{N}(P)$ is isomorphic to a proper subgraph of~$L(\comp{G})$.
    Since $\mathcal{N}(P^*)\cong L(\comp{G})$, this means that $N_P<N_{P^*}$, which means that $P\notin \Omega_4$.
    So for every $P\in \Omega_4$, we know that~$\psi_P$ is bijective, and that if~$\psi(Q_1)$ and~$\psi(Q_2)$ are incident, then $Q_1Q_2\in E(\mathcal{B})$.
    In particular, $\mathcal{N}(P)\cong L(\comp{G})$ for every $P\in \Omega_4$ by definition of a line graph.
\end{proof}

We now rule out a case where certain vertices $P\in \Omega_3$ do not satisfy $\mathcal{N}(P)\cong L(\comp{G})$.

\begin{lemma}\label{lem:notype1with34}
    Let $P\in \Omega_{3}$, and let $Q_1,Q_2\in N(P)$ be neighbours of~$P$.
    If~$\psi(Q_1)$ and~$\psi(Q_2)$ are incident and $\type(Q_1)=1$ and $\type(Q_2)\in \{3,4\}$, then we have $P\notin \Omega_4$.
\end{lemma}

\begin{proof}
    We claim that~$Q_1$ and~$Q_2$ are not adjacent in~$\mathcal{B}$.
    In this case, we know that~$\mathcal{N}(P)$ is a proper subgraph of~$L(\comp{G})$ by Lemma~\ref{lem:linegraph}, so  $P\notin \Omega_4$ by Lemma~\ref{lem:linegraph2}.

    It remains to prove the claim.
    Let~$u$ denote the common vertex of~$\psi(Q_1)$ and~$\psi(Q_2)$.
    Then~$Q_1$ is obtained from~$P$ by merging the part $\{u\}\in P$ with another part $\{v\}\in P$, and~$Q_2$ is obtained from~$P$ by adding a vertex~$w$ from a part $\{w,x\}\in P$ to~$\{u\}$.
    Then we have $\{u,v\},\{w,x\}\in Q_1$, while $\{u,w\}\in Q_2$.
    So~$Q_1$ and~$Q_2$ are not adjacent in~$\mathcal{B}$.
\end{proof}

Lemma~\ref{lem:linegraph2} allows us to determine the line graph $L(\comp{G})$ from $\mathcal{B}$.
It remains to show that we can distinguish between claws and triangles in~$\comp{G'}$.
In order to do this, we will examine the neighbourhood of the vertices $P\in \Omega_4$ at distance~$2$ from~$P$, as discussed in the outline at the beginning of this subsection.

\begin{lemma}\label{lem:P^*dist2}
    Let $P\in \Omega_{4}$.
    Let $\{Q_1,Q_2,Q_3\}\subseteq N(P)$ be the vertex set of a triangle in~$\mathcal{B}$.
    If $\psi(Q_1)$,~$\psi(Q_2)$, and~$\psi(Q_3)$ are the three edges of a claw in~$\comp{G}$, then $Q_1$,~$Q_2$, and~$Q_3$ do not have any common neighbour outside of~$N[P]$.
\end{lemma}

\begin{proof}
    Without loss of generality, we may assume that $\type(Q_1)\leq \type(Q_2)\leq \type(Q_3)$.

    If $\type(Q_1)=\type(Q_2)=\type(Q_3)=1$, then the result follows immediately from Lemma~\ref{lem:type1neighbours}.
    By Lemma~\ref{lem:notype1with34}, if we have $\type(Q_1)=1$ and either $\type(Q_2)=4$ or $\type(Q_3)=4$, then we know that $P\notin \Omega_4$, a contradiction.
    If we have $\type(Q_1)=\type(Q_2)=\type(Q_3)=5$, then by definition, $\psi(Q_1)$, $\psi(Q_2)$, and~$\psi(Q_3)$ are the three edges of a triangle in~$\comp{G}$, a contradiction.

    Therefore, by Lemma~\ref{lem:pairsoftypes}, it remains only to consider the case where $\type(Q_1)\in\{2,3,4\}$ and $\type(Q_2),\type(Q_3)\in \{3,4\}$.
    Then,~$Q_2$ is obtained from~$P$ by adding a vertex~$u_2$ from a part $\{u_2,v_2\}\in P$ to a part $\{w_2\}\in P$, so that $\psi(Q_2)=\{u_2,w_2\}$.
    Let $Q'\in N(P)\cap N(Q_2)$ be any common neighbour of~$P$ and~$Q_2$ such that $\type(Q')\in \{3,4\}$.
    Then~$Q'$ is obtained from~$P$ by adding a vertex~$u'$ from a part $\{u',v'\}\in P$ to a part $\{w'\}\in P$, so that $\psi(Q')=\{u',w'\}$.
    By Lemma~\ref{lem:linegraph},~$\psi(Q')$ shares a vertex with~$\psi(Q_2)$, so either $u_2=u'$ or $w_2=w'$.
    Suppose that $w_2=w'$ and $\psi_P(Q')\neq \{v_2,w_2\}$.
    Then $\{u_2,w_2\}$ is a part of~$Q_2$, whereas~$u_2$ and~$w_2$ belong to distinct parts of $Q'$, each of size~$2$.
    So~$Q_2$ is not adjacent to~$Q'$, a contradiction.
    So either $u_2=u'$, or both $w_2=w'$ and $\psi_P(Q')=\{v_2,w_2\}$.
    Since $\psi(Q_1)$,~$\psi(Q_2)$, and~$\psi(Q_3)$ are distinct, this means that if~$w_2$ is the common vertex of $\psi(Q_1)$,~$\psi(Q_2)$, and~$\psi(Q_3)$, then we cannot have both $\type(Q_1),\type(Q_3)\in \{3,4\}$.
    Therefore we have $\type(Q_1)=2$ in this case, a contradiction as $\{w_2\}\in P$ is a part of size~$1$.
    So~$u_2$ is the common vertex of $\psi(Q_1)$,~$\psi(Q_2)$, and~$\psi(Q_3)$.
    So~$Q_3$ is obtained from~$P$ by adding~$u_2$ to some part $\{w_3\}\in P$, where $w_3\neq w_2$.

    Now suppose $R\notin N[P]$ is a common neighbour of~$Q_2$ and~$Q_3$.
    Since $R\notin N[P]$,~$R$ is not obtained from~$Q_2$ or~$Q_3$ by moving~$u_2$ by Observation~\ref{obs:vopenneighbourhoodclique}.
    If any of $Q_1$,~$Q_2$, or~$Q_3$ have type~$4$, then $\ord{P}=k$.
    Otherwise, $\type(Q_2)=\type(Q_3)=3$, so~$w_2$ and~$w_3$ each have degree $n-2$ in~$G$, and so are adjacent to each other in~$G$.
    So in either case,~$R$ is not obtained from~$Q_2$ by adding~$w_3$ to $\{u_2,w_2\}$, and~$R$ is not obtained from~$Q_3$ by adding~$w_2$ to $\{u_2,w_3\}$.
    So $R(w_2)\neq R(u_2)$ and $R(w_3)\neq R(u_2)$.
    So~$R$ is obtained from~$Q_2$ by moving~$w_2$, and~$R$ is obtained from~$Q_3$ by moving~$w_3$.
    The former statement means that $\{w_3\}\in R$, so by the latter statement,~$R$ is obtained from~$Q_3$ by splitting~$w_3$.
    But this is a contradiction, since then~$R$ is obtained from~$P$ by splitting~$u_2$, meaning $R\in N[P]$.
    So no such~$R$ exists.
\end{proof}

So far, every proof in this section holds for $\mathcal{B}\in \{\mathcal{B}(G),\mathcal{B}_{\geq k}(G)\}$ with $k\leq n-1$.
However, for the rest of this section, we will further suppose that $k\leq n-2$.

\begin{lemma}\label{lem:P^*dist2_2}
    Suppose $k\leq n-2$.
    Let $\{Q_1,Q_2,Q_3\}\subseteq N(P^*)$ be the vertex set of a triangle in~$\mathcal{B}$.
    If $\psi(Q_1)$, $\psi(Q_2)$, and $\psi(Q_3)$ are the three edges of a triangle in~$\comp{G}$, then $Q_1$,~$Q_2$, and~$Q_3$ have a common neighbour $R\notin N[P^*]$.
\end{lemma}

\begin{proof}
    If $\psi(Q_1)$,~$\psi(Q_2)$, and~$\psi(Q_3)$ form the three edges of a triangle, then let $\{u,v,w\}$ be the vertex set of this triangle, and let~$R$ be the partition obtained from~$P^*$ by replacing $\{u\}$,~$\{v\}$, and~$\{w\}$ with $\{u,v,w\}$.
    Then~$R$ is a neighbour of $Q_1$,~$Q_2$, and~$Q_3$, and $R\notin N[P^*]$.
\end{proof}

The following lemma completes the proof of Theorem~\ref{thm:maincomplete}, from which Theorem~\ref{thm:main} follows immediately.

\begin{lemma}\label{lem:mainreconstruction}
    Suppose $k\leq n-2$.
    Then we can determine~$G'$ from~$\mathcal{B}$.
\end{lemma}

\begin{proof}
    Since $P^*\in \Omega_4$ by Lemma~\ref{lem:linegraph2}, we know that~$\Omega_4$ is non-empty.
    Therefore~$\Omega_5$ is non-empty by definition of~$\Omega_5$.
    Observe that we can recognise the vertices in~$\Omega_5$ by direct inspection of~$\mathcal{B}$.
    Let $P\in \Omega_5$ be given.
    Since $\Omega_5\subseteq \Omega_4$, by Lemma~\ref{lem:linegraph2} we know that~$\mathcal{N}(P)$ is isomorphic to the line graph~$L(\comp{G})$.
    By Observation~\ref{obs:lineinverse}, we can therefore recognise the graph~$\compgddagger$ obtained from~$\comp{G}$ by removing all of its isolated vertices and by replacing every triangle component with a claw component.
    Let~$T$ denote the number of triangles in~$\comp{G}$, and let~$T^{\ddagger}$ denote the number of triangles in~$\compgddagger$.
    Observe that~$\comp{G'}$ is the graph obtained from~$\comp{G}$ by removing all of its isolated vertices, since a vertex is universal in~$G$ if and only if it is isolated in~$\comp{G}$.
    Therefore,~$\comp{G'}$ is the graph obtained from~$\compgddagger$ by replacing $T-T^{\ddagger}$ of its claw components with triangle components.
    Clearly we can determine~$T^{\ddagger}$ from~$\compgddagger$.
    Suppose that we also have a strategy to determine~$T$.
    Then we can determine~$\comp{G'}$, and therefore we can determine~$G'$, completing the proof.

    It remains to show that we can determine~$T$.
    Let $P\in \Omega_4$ be fixed.
    We claim that if $P\in \Omega_5$, then $T_P=T$.
    If so, we are done since we can recognise~$\Omega_5$, and since we can determine~$T_{P'}$ from~$\mathcal{B}$ for each $P'\in \Omega_5$.
    Let~$\mathcal{T}_P$ denote the set of triangles in~$\mathcal{N}(P)$.
    Let~$X_T$ denote the set of triangles in~$\comp{G}$, let~$X_C$ denote the set of claws in~$\comp{G}$, and let $X\coloneqq X_T\cup X_C$.
    By Lemma~\ref{lem:linegraph2}, for each triangle or claw $K\in X$ with edges $e_1$,~$e_2$, and~$e_3$, there exist distinct neighbours $Q_1$,~$Q_2$, and~$Q_3$ of~$P$ such that $\psi(Q_i)=e_i$ for each $i\in \{1,2,3\}$, and such that $\{Q_1,Q_2,Q_3\}$ forms a triangle in~$\mathcal{N}(P)$.
    On the other hand, let $\Delta\in \mathcal{T}_P$ be a triangle on vertex set $\{Q_1,Q_2,Q_3\}$.
    Then, again by Lemma~\ref{lem:linegraph2}, we know that $\psi_P(Q_1)$,~$\psi_P(Q_2)$, and~$\psi_P(Q_3)$ are pairwise incident as edges of~$\comp{G}$, and therefore are either the three edges of a triangle in~$\comp{G}$ or the three edges of a claw in~$\comp{G}$.
    Therefore, there is a bijection $\xi:\mathcal{T}_P\rightarrow X$ which sends each triangle~$\Delta$ in~$\mathcal{N}(P)$ on vertex set $Q_1$,~$Q_2$, and~$Q_3$ to the triangle or claw in~$\comp{G}$ with edges $\psi(Q_1)$,~$\psi(Q_2)$, and~$\psi(Q_3)$.
    Moreover, Lemma~\ref{lem:P^*dist2} tells us that if $\xi(\Delta)\in X_C$, then the three vertices of~$\Delta$ have no common neighbour outside of~$N[P]$.
    So therefore $T_P\leq \ord{X_T}=T$.
    On the other hand, in the special case where~$P=P^*$, Lemma~\ref{lem:P^*dist2_2} tells us that if $\xi(\Delta)\in X_T$, then the three vertices of~$\Delta$ do have a common neighbour outside of~$N[P]$.
    So we have $T_{P^*}\geq \ord{X_T}=T$, which means that $T_{P^*}=T$.
    By definition of~$\Omega_5$, this means that $T_P=T$ for every $P\in \Omega_5$, as claimed.
\end{proof}

\subsection{Completing the proof of Theorem~\ref{thm:mainupperhalfcomplete}}\label{subsect:completingupper}

In this brief subsection, we will complete the proof of Theorem~\ref{thm:mainupperhalfcomplete}, from which Theorem~\ref{thm:mainupper} immediately follows.
Most of the work has already been completed.
It remains to examine how the upper Bell $k$-colouring graphs~$\mathcal{B}_{\geq k}(G)$ change as we vary~$k$, or remove universal vertices from~$G$.

The following observation follows immediately from our previous observation that in any independent set partition of any graph, any universal vertex must belong to a part of size~$1$.

\begin{observation}\label{obs:addfulldegree}
    Let~$G^+$ be the graph obtained from~$G$ by adding a universal vertex. Then $\chi(G^+)=\chi(G)+1$ and $\mathcal{B}_{\geq k}(G)\cong \mathcal{B}_{\geq k+1}(G^+)$.
\end{observation}

\begin{lemma}\label{lem:reconfigsize}
    Let~$k_1$ and~$k_2$ be natural numbers. If $\mathcal{B}_{\geq k_1}(G)\cong \mathcal{B}_{\geq k_2}(G)$, then either $k_i > n$ for both $i\in \{1,2\}$, or $\chi(G)<k_i\leq n$ for both $i\in \{1,2\}$ and $k_1=k_2$, or $k_i \leq \chi(G)$ for both $i\in \{1,2\}$.
\end{lemma}

\begin{proof}
    For each natural number~$k$, there exists a partition of~$V(G)$ into exactly~$k$ independent sets if and only if $k\in [\chi(G),n]$.
    Hence $\mathcal{B}_{\geq k}(G)=\mathcal{B}_{\geq \chi(G)}(G)$ if $k\leq \chi(G)$, and $\mathcal{B}_{\geq k}(G)=\mathcal{B}_{\geq n+1}(G)$ if $k>n$, and every graph in $\{\mathcal{B}_{\geq \chi(G)}(G),\ldots,\mathcal{B}_{\geq n+1}(G)\}$ is distinct as these graphs each have a different number of vertices to each other.
\end{proof}

\begin{lemma}\label{lem:uppermainuseful}
    Let~$G_1$ and~$G_2$ be graphs on~$n_1$ and~$n_2$ vertices respectively.
    Let~$k_1$ and~$k_2$ be natural numbers.
    Suppose $n_1-k_1=n_2-k_2$ and $G_1'\cong G_2'$.
    Then $\mathcal{B}_{\geq {k_1}}(G_1)\cong\mathcal{B}_{\geq {k_2}}(G_2)$.
\end{lemma}

\begin{proof}
    Let~$n'$ denote the number of vertices of~$G_1'$ (or~$G_2'$).
    For each $i\in\{1,2\}$,~$G_i$ is obtained from~$G_i'$ by adding $n_i-n'$ universal vertices.
    Hence, by Observation~\ref{obs:addfulldegree}, we have $\mathcal{B}_{\geq k_1-(n_1-n')}(G_i')\cong\mathcal{B}_{\geq k_i}(G_i)$ for each $i\in\{1,2\}$.
    Moreover, since $n_1-k_1=n_2-k_2$, we know that $k_1-(n_1-n')=k_2-(n_2-n')$, and therefore we have $\mathcal{B}_{\geq k_1-n_1+n'}(G_1')\cong \mathcal{B}_{\geq k_2-n_2+n'}(G_2')$.
    Therefore we have $\mathcal{B}_{\geq {k_1}}(G_1)\cong\mathcal{B}_{\geq {k_2}}(G_2)$, as claimed.
\end{proof}

We now complete the proof of Theorem~\ref{thm:mainupperhalfcomplete}.

\begin{proof}[Proof of Theorem~\ref{thm:mainupperhalfcomplete}]
    For the `if' direction, suppose $G_1'\cong G_2'$.
    If $n_1-k_1=n_2-k_2$, then we have $\mathcal{B}_{\geq k_1}(G_1)\cong\mathcal{B}_{\geq k_2}(G_2)$ by Lemma~\ref{lem:uppermainuseful} and we are done.
    If $k_i\leq \chi(G_i)$ for both $i\in \{1,2\}$, then $\mathcal{B}_{\geq k_i}(G_i)=\mathcal{B}(G_i)$ for both $i\in \{1,2\}$.
    But we know that $\mathcal{B}(G_1')\cong \mathcal{B}(G_2')$, so we have $\mathcal{B}(G_1)\cong\mathcal{B}(G_2)$ by Theorem~\ref{thm:maincomplete}.
    Therefore, we have $\mathcal{B}_{\geq k_1}(G_1)\cong\mathcal{B}_{\geq k_2}(G_2)$.

    For the `only if' direction, if $\mathcal{B}_{\geq k_1}(G_1)\cong\mathcal{B}_{\geq k_2}(G_2)$, we have $G_1'\cong G_2'$ by Lemma~\ref{lem:mainreconstruction}.
    Let~$n'$ denote the number of vertices of~$G_1'$ (or~$G_2'$).
    Since we have $\mathcal{B}_{\geq k_1}(G_1)\cong\mathcal{B}_{\geq k_2}(G_2)$, we have
    $\mathcal{B}_{\geq k_1-(n_1-n')}(G_1')\cong\mathcal{B}_{\geq k_2-(n_2-n')}(G_2')$ by Observation~\ref{obs:addfulldegree}.
    Since $G_1'\cong G_2'$, Lemma~\ref{lem:reconfigsize} guarantees that one of the following three scenarios holds.
    Either we have $k_i-(n_i-n')>n_i$ for both $i\in \{1,2\}$, which means that $k_i>n_i$ for both $i\in \{1,2\}$, a contradiction.
    Or, we have $k_1-(n_1-n')=k_2-(n_2-n')$, which means that $n_1-k_1=n_2-k_2$, as claimed.
    Or, we have $k_i-(n_i-n')\leq \chi(G_i')$ for both $i\in \{1,2\}$, which means $k_i\leq \chi(G_i)$ for both $i\in \{1,2\}$ by Observation~\ref{obs:addfulldegree}, as claimed.
\end{proof}

%% file: 3_lowertheorem.tex
\section{Proof of Theorem~\ref{thm:mainlower}}

In this section we will prove Theorem~\ref{thm:mainlower}.
Fix a graph~$G$ on~$n$ vertices and a natural number~$k$.
We use the same terminology introduced at the start of Section~\ref{section:main}. Our proof will not use any results from Section~\ref{section:main} (except for Observations~\ref{obs:vopenneighbourhoodclique} and~\ref{obs:identicalneighbours}), instead expanding on ideas from Asgarli et al.~\cite{asgarli2025coloring}, Berthe et al.\ \cite{berthe2025determining}, and Hogan et al.\ \cite{hogan2024note}, where the local structure of a recolouring graph~$\mathcal{C}_{k}(G)$ is used to determine the underlying graph~$G$.
Unfortunately, the local structure of a Bell $k$-colouring graph is considerably more complex.
In the case of recolouring graphs, it is straightforward to show that the neighbourhood of any partition~$P$ consists of disjoint cliques, each corresponding to recolouring a given vertex of the underlying graph.
However this is not necessarily true for Bell colouring graphs if~$P$ contains parts of size less than~$4$.
In particular, it is true that any two partitions obtained from~$P$ by moving the same vertex of~$G$ are adjacent to each other (this is Observation~\ref{obs:vopenneighbourhoodclique}), however observe that if $\{u,v\}\in P$ has size~$2$, then the $u$-split neighbour and the $v$-split neighbour of~$P$ are equal, and that if $\{u,v,w\}\in P$ is a part of size~$3$, then the $u$-split, $v$-split, and $w$-split neighbours of~$P$ are each adjacent to each other.
Nonetheless, if we can guarantee the existence of a partition $P^\dagger\in V(\BlG)$ where every part of~$P^\dagger$ has size at least~$4$, then so long as $\ord{P^\dagger}<k$, it will be possible with some care to use a similar approach to Berthe et al.\ to recover the underlying graph~$G$.

In particular, our approach for proving Theorem~\ref{thm:mainlower} is as follows.
For the rest of this section, suppose~$G$ has maximum degree $\Delta(G)<\frac{1}{9}n-\frac{1}{3}$, and suppose $k\geq \chi(G)+1$.
We will show that this maximum degree condition on~$G$ guarantees the existence of a partition $P^\dagger\in V(\BlG)$ such that $\ord{P^\dagger}=\chi(G)$ and every part of~$P^\dagger$ has size at least~$4$ (see Lemma~\ref{lem:chihajnal}).
By Observation~\ref{obs:vopenneighbourhoodclique}, for every partition $P\in V(\BlG)$, the open neighbourhood~$\mathcal{N}(P)$ has at most~$n$ components, each corresponding to at least one vertex of~$G$.
In particular, we will show that~$\mathcal{N}(P^\dagger)$ has exactly~$n$ components, each corresponding to a unique vertex of~$G$.
We will then show that we can recognise whether we are in one of two regimes: the first where $k=\chi(G)+1$, the second where $k\geq \chi(G)+2$.
Then, looking only at partitions $P\in V(\BlG)$ whose neighbourhoods contain a maximal number of components (i.e.~$n$ components), we will construct a candidate graph~$G_P$ for each such~$P$ by looking at the neighbourhood in~$\BlG$ at distance~$2$ from~$P$.
(Note that this construction will differ depending on which regime we are in.)
We will then show that~$G_{P^\dagger}$ is isomorphic to~$G$, whereas every other candidate graph is isomorphic to either a subgraph of~$G$ (if we are in the first regime) or a supergraph of~$G$ (if we are in the second regime).
This allows us to reconstruct~$G$ by selecting either a largest candidate graph or a smallest candidate graph, depending on the regime.

We first show the following, related to the classical theorem of Hajnal and Szemerédi on vertex partitions into independent sets \cite{hajnal1970proof}.

\begin{lemma}\label{lem:chihajnal}
    Every graph~$G$ with maximum degree $\Delta(G)<\frac{1}{9}n-\frac{1}{3}$ has a vertex partition into~$\chi(G)$ independent sets, each of size at least~$4$.
\end{lemma}

\begin{proof}
    Let $\Delta\coloneqq \Delta(G)$.
    For each partition~$P$ of~$V(G)$ into~$\chi(G)$ independent sets, let $m\coloneqq m(P)$ be the size of a smallest part in~$P$, and let $c\coloneqq c(P)$ be the number of such parts of minimal size.
    Choose a partition~$P$ of~$V(G)$ into~$\chi(G)$ independent sets such that~$c$ is minimal among the set of such partitions where~$m$ is maximal.
    For each natural number~$\ell$, denote by~$n_\ell$ (respectively~$n_\ell^+$) the number of vertices in~$G$ belonging to a part of size~$\ell$ (resp.\ of size at least~$\ell$) in~$P$.

    If $m\geq 4$ then we are done.
    So suppose $m\leq 3$.
    Let $A,B\in P$ be distinct, where~$A$ has size~$m$.
    Then there is at least one edge between~$A$ and~$B$, since otherwise replacing~$A$ and~$B$ with $A\cup B$ would yield a partition of~$V(G)$ into $\chi(G)-1$ independent sets.

    Suppose~$B$ has size $m+1$.
    Suppose that there is exactly one edge~$ab$ between~$A$ and~$B$.
    Let $C\in P$ be distinct from~$A$ and~$B$.
    Then~$C$ contains at least one neighbour of~$a$, since otherwise adding~$a$ to~$C$ and adding all other vertices of~$A$ to~$B$ would give a partition of~$V(G)$ into $\chi(G)-1$ independent sets.
    Moreover, if~$C$ has size $\ell\geq 2m+1$, then~$C$ must contain at least $\ell-m+1$ neighbours of~$a$, since if there are~$m$ vertices $c_1,\ldots,c_m\in C$ not adjacent to~$a$, then the partition obtained from~$P$ by replacing $A$,~$B$, and~$C$ with $\{a,c_1,\ldots,c_m\}$, $B\cup A\setminus \{a\}$, and $C\setminus \{c_1,\ldots,c_m\}$ has fewer parts of size~$m$ than~$P$, contradicting the minimality of~$P$.
    Therefore we can calculate
    \begin{align*}
        \Delta\geq d(a)\geq \frac{n_m-m}{m} + \sum_{\ell=m+1}^{2m}\frac{n_\ell}{\ell}+\sum_{\ell=2m+1}^\infty\frac{(\ell-m+1)n_\ell}{\ell}\geq \frac{n-m}{2m},
    \end{align*}
    a contradiction since $m\leq 3$.
    Therefore there are at least two edges between~$A$ and~$B$ in this case.
    
    Suppose~$B$ has size at least $m+2$.
    Suppose there is a vertex~$b\in B$ with no neighbours in~$A$.
    Then the partition~$Q$ obtained from~$P$ by adding~$b$ to~$A$ would satisfy either $c(Q)<c(P)$ or $m(Q)>m(P)$, contradicting our choice of~$P$ in either case.
    Therefore, in this case, every $b\in B$ is adjacent to at least one vertex in~$A$.
    
    Therefore we can calculate
    \begin{align*}
        m\Delta\geq \sum_{a\in A}d(a)\geq \frac{n_m-m}{m} + \frac{2n_{m+1}}{m+1}+n_{m+2}^+\geq \frac{n-m}{m},
    \end{align*}
    a contradiction since $m\leq 3$.
\end{proof}

We will use the following terminology and notation throughout the rest of this section.
For each partition $P\in V(\BlG)$, denote by~$C_P$ the number of components of~$\mathcal{N}(P)$.
Call a partition $P\in V(\BlG)$ a \textit{reconstruction candidate} if $C_P=\max_{P'\in V(\BlG)}C_{P'}$, i.e.\ if the open neighbourhood of~$P$ has as at least as many components as any other partition of~$\BlG$.
For each $P\in V(\BlG)$ and $v\in V(G)$, define the \textit{$v$-neighbourhood} $N_v(P)\subseteq N(P)$ of~$P$ to be set of vertices $Q\in N(P)$ which can be obtained from~$P$ by moving~$v$.

\begin{lemma}\label{lem:reconstructioncandidatefacts}
    If $P^\dagger\in V(\mathcal{B})$ is a partition with $\chi(G)$ parts, where each part has size at least~$4$, then~$P^\dagger$ is a reconstruction candidate.
    Conversely, each of the following statements hold for every reconstruction candidate $P\in V(\BlG)$.
    \begin{enumerate}
        \item We have $C_P=n$.
        \item\label{lem:reconstructioncandidatefacts2} Let $A\in P$ be a part of size $1$ and let $B\in P$ be a part of size at most $2$.
        Then there is at least one edge $uv\in E(G)$ such that $u\in A$ and $v\in B$.
        \item Either we have $\ord{P}=k$, or~$P$ has no parts of size~$2$ or~$3$.
    \end{enumerate}
\end{lemma}

\begin{proof}
    Let~$P^\dagger\in V(\BlG)$ be a partition into~$\chi(G)$ parts, each of size at least~$4$.
    Since $k>\chi(G)$, for each $u\in V(G)$, we know that~$N_u(P^\dagger)$ is non-empty, since the $u$-split neighbour of~$P^\dagger$ belongs to~$N_u(P^\dagger)$.
    Let $u,v\in V(G)$ be distinct and let $Q_u\in N_u(P^\dagger)$ and $Q_v\in N_v(P^\dagger)$.
    Since every part of~$P^\dagger$ has size at least~$4$, there are at least two vertices $u',u''\in P^\dagger(u)$ distinct from both~$u$ and~$v$.
    Observe that $Q_u(u)\neq Q_u(u')=Q_u(u'')$, whereas $Q_v(u)=Q_v(u')=Q_v(u'')$.
    Therefore, we know that $Q_u\neq Q_v$, and moreover if~$Q_v$ is adjacent to~$Q_u$, then~$Q_v$ is obtained from~$Q_u$ by adding~$u$ to~$Q_u(u')$, which means that $Q_v=P^\dagger$, a contradiction.
    Therefore,~$Q_u$ and~$Q_v$ are distinct and not adjacent.
    It follows that $C_{P^\dagger}=n$.
    Observe that $C_Q\leq n$ for every $Q\in V(\mathcal{B})$ by Observation~\ref{obs:vopenneighbourhoodclique}, which applies for $\mathcal{B}=\BlG$ since~$\BlG$ is an induced subgraph of~$\BG$.
    Therefore,~$P^\dagger$ is a reconstruction candidate.
    Moreover, by Lemma~\ref{lem:chihajnal}, such a $P^\dagger$ does exist, and so $C_P=n$ for any reconstruction candidate $P\in V(\mathcal{B})$.
    
    For the final two statements, it suffices to show that if either condition fails, then there exist distinct $u,v\in  V(G)$ such that either $N_u(P)\cap N_v(P)\neq \emptyset$, or there is an edge from a vertex in~$N_u(P)$ to a vertex in~$N_v(P)$.
    If~$P$ has two parts~$\{u\}$ and~$\{v\}$ such that $uv\notin E(G)$, then the partition~$Q_{uv}$ obtained by merging~$\{u\}$ and~$\{v\}$ satisfies $Q_{uv}\in N_u(P)\cap N_v(P)$.
    If~$P$ has two parts~$\{u\}$ and $\{v,w\}$ such that~$uv,uw\notin E(G)$, then the partition~$Q_{vu}$ obtained by adding~$v$ to~$\{u\}$ and the partition~$Q_{wu}$ obtained by adding~$w$ to~$\{u\}$ are distinct and adjacent.
    Now suppose~$P$ has at most $k-1$ parts.
    For each $u\in V(G)$ such that $\ord{P(u)}\geq 2$, let~$Q_u$ denote the $u$-split neighbour of~$P$.
    If~$P$ has a part $\{u,v\}$ of size~$2$, then we have $Q_u=Q_v$, so therefore $Q_u\in N_u(P)\cap N_v(P)$.
    If~$P$ has a part $\{u,v,w\}$ of size $3$, then the two partitions $Q_u\in N_u(P)$ and $Q_v\in N_v(P)$ are distinct, and are adjacent to each other in~$\BlG$ since each can be obtained from the other by moving~$w$.
\end{proof}

As discussed at the start of this section, our goal now is to examine the neighbourhood at distance~$2$ of each reconstruction candidate $P\in V(\BlG)$, in order to build a candidate graph~$G_P$ for each such~$P$.
In particular, given two vertices $u,v\in V(G)$, in order to determine whether or not~$u$ and~$v$ should be adjacent in~$G_P$, we will consider pairs of neighbours $Q_u\in N_u(P)$ and $Q_v\in N_v(P)$, and look at the set of common neighbours of~$Q_u$ and~$Q_v$ which lie outside of~$N[P]$.

Let $P\in V(\BlG)$ be given, and let $Q_1,Q_2\in N(P)$ be neighbours of~$P$.
Say that~$Q_1$ and~$Q_2$ are \textit{double-closed} with respect to~$P$ if they are not adjacent to each other in~$\BlG$, and if there exist exactly two vertices of $N(Q_1)\cap N(Q_2)$ which are not in~$N[P]$ and which are adjacent to each other. (Observe that there may still be other vertices in $N(Q_1)\cap N(Q_2)$.)
For ease of writing, we will not specify that such a pair are double-closed with respect to~$P$ when the partition~$P$ is clear from context.

We now begin by studying certain split neighbours of a given partition $P\in V(\BlG)$.

\begin{lemma}\label{lem:splitclosureinformation}
    Let $P\in V(\BlG)$ be a partition of size at most $k-1$.
    Let $u,v\in V(G)$ be distinct such that $\ord{P(u)},\ord{P(v)}\geq 4$.
    Let~$Q_u$ and~$Q_v$ be the $u$-split neighbour and $v$-split neighbour of~$P$ respectively.
    Then the following statements hold.
    \begin{enumerate}
        \item $Q_u$ and~$Q_v$ are double-closed if and only if $uv\notin E(G)$ and $\ord{P}\leq k-2$.
        \item Every common neighbour of $Q_u$ and~$Q_v$ belongs to~$N[P]$ if and only if $uv\in E(G)$ and $\ord{P}=k-1$.
    \end{enumerate}
\end{lemma}

\begin{proof}
    First, observe that $\{u\}\in Q_u$ and $\ord{Q_v(u)}\geq 3$.
    Therefore, if~$Q_u$ and~$Q_v$ are adjacent in~$\BlG$, then~$Q_u$ is the $u$-split neighbour of~$Q_v$, a contradiction since $\{v\}\in Q_v$ whereas $\{v\}\notin Q_u$.
    So~$Q_u$ and~$Q_v$ are not adjacent in~$\BlG$.
    
    We will now show that there are only four possible common neighbours of~$Q_u$ and~$Q_v$ outside of~$N[P]$.
    The result will then follow by examining when these common neighbours exist as vertices of~$\BlG$, and how they relate to each other if so.
    
    Suppose~$Q_u$ and~$Q_v$ have a common neighbour $R\notin N[P]$.
    Observe that we have $\ord{Q_u(v)}\geq 3$.
    If we have $\ord{R(v)}\geq 3$, then since $\{v\}\in Q_v$, we know that~$R$ is obtained from~$Q_v$ by moving~$v$.
    But this gives a contradiction, as then we have $R\in N[P]$ by Observation~\ref{obs:vopenneighbourhoodclique}.
    So we have $\ord{R(v)}\leq 2$, which means that~$R$ is obtained from~$Q_u$ by moving a vertex $x\in Q_u(v)$.
    
    Suppose $x=v$.
    If~$R$ is obtained from~$Q_u$ by adding~$v$ to a part $A\in P$ other than~$\{u\}$, then we have $\{u\}\in R$ and $\ord{Q_v(u)}\geq 3$, so~$R$ is the $u$-split neighbour of~$Q_v$.
    But then $A\cup\{v\}\notin R$, a contradiction.
    So~$R$ is obtained from~$Q_u$ either by splitting~$v$, or by adding~$v$ to~$\{u\}$.
    Denote the two partitions obtained from~$Q_u$ in this way by~$R_1$ and~$R_2$ respectively, observing that they may or may not exist as vertices in~$\BlG$.
    Indeed, observe that $R_1\in V(\BlG)$ if and only if $\ord{P}\leq k-2$, and observe that $R_2\in V(\BlG)$ if and only if $uv\notin E(G)$.
    Moreover, observe that if both $R_1,R_2\in V(\BlG)$, then~$R_1$ is the $u$-split neighbour of~$R_2$, i.e.~$R_1$ and $R_2$ are adjacent in $\BlG$.
    
    Suppose $x\neq v$.
    Then, since $\ord{R(v)}\leq 2$, we know that $\ord{P(v)}=4$ and $u,v\in P(v)$.
    In particular, we have $uv\notin E(G)$ in this case.
    Let $P(v)=\{u,v,a,b\}$ for some vertices $a,b\in V(G)$.
    Then we have $x\in \{a,b\}$, and we either have $R(v)=\{v,a\}$, or $R(v)=\{v,b\}$.
    Both~$\{v\}$ and $\{u,a,b\}$ are parts of~$Q_v$, so~$R$ is obtained from~$Q_v$ by adding the unique vertex of $R(v)\setminus\{v\}$ to~$\{v\}$.
    This means that~$R$ is uniquely determined by our choice of $x\in \{a,b\}$.
    Let~$R_a$ and~$R_b$ respectively denote this partition~$R$ in the case where $x=a$ and $x=b$, observing that~$R_a$ and~$R_b$ may or may not exist as vertices of~$\BlG$.
    Recall in particular that if $R_a$ and $R_b$ are vertices of $\BlG$, then $uv\notin E(G)$.
    Moreover, observe that if~$R$ and~$R'$ are distinct vertices from $\{R_a,R_b,R_1,R_2\}$ such that $\{R,R'\}\neq \{R_1,R_2\}$, and if $R,R'\in V(\BlG)$, then~$R$ and~$R'$ are not adjacent in~$\BlG$, since~$R$ contains a part $\{y,z\}$ such that~$y$ and~$z$ belong to distinct parts of~$R'$, each of size~$2$.

    It follows that~$Q_u$ and~$Q_v$ are double-closed if and only if~$R_1$ and~$R_2$ are vertices of~$\BlG$, which we have shown holds if and only if $uv\notin E(G)$ and $\ord{P}\leq k-2$.
    For the second claim, we have shown that if $uv\in E(G)$ and $\ord{P}=k-1$, then none of $R_a$,~$R_b$,~$R_1$, or~$R_2$ exist as vertices of~$\BlG$, and so every common neighbour of $Q_u$ and $Q_v$ belongs to $N[P]$.
    On the other hand, we have also shown that if $uv\notin E(G)$, then $R_2\in V(\BlG)\setminus N[P]$, and if $\ord{P}\leq k-2$, then $R_1\in V(\BlG)\setminus N[P]$, so there is a common neighbour of $Q_u$ and $Q_v$ outside of $N[P]$ in either case.
\end{proof}

The following lemma will allow us to use the property of being double-closed in order to detect edges of the underlying graph~$G$.

\clearpage

\begin{lemma}\label{lem:doubleclosureinformation}
    Let $P\in V(\BlG)$ be a reconstruction candidate.
    Let $u,v\in V(G)$.
    If there exist $Q_u\in N_u(P)$ and $Q_v\in N_v(P)$ which are double-closed, then $uv\notin E(G)$ and $\ord{P}\leq k-2$.
    Moreover, if every part of~$P$ has size at least~$4$, then if $uv\notin E(G)$ and $\ord{P}\leq k-2$, then there exist $Q_u\in N_u(P)$ and $Q_v\in N_v(P)$ which are double-closed.
\end{lemma}

\begin{proof}
    Suppose every part of~$P$ has size at least~$4$, and suppose $uv\notin E(G)$ and $\ord{P}\leq k-2$.
    Let~$Q_u$ be the $u$-split neighbour of~$P$ and let~$Q_v$ be the $v$-split neighbour of~$P$.
    Then~$Q_u$ and~$Q_v$ are double-closed by Lemma~\ref{lem:splitclosureinformation}.
    
    Now suppose there exist $Q_u\in N_u(P)$ and $Q_v\in N_v(P)$ which are double-closed.
    
    Suppose~$Q_u$ and~$Q_v$ are both split neighbours of~$P$.
    Then~$P$ has at most $k-1$ parts, so since~$P$ is a reconstruction candidate,~$P(u)$ and~$P(v)$ have size at least~$4$.
    By Lemma~\ref{lem:splitclosureinformation}, we therefore have both $uv\notin E(G)$ and $\ord{P}\leq k-2$.
    
    Suppose~$Q_u$ and~$Q_v$ are not both split neighbours of~$P$.
    Without loss of generality, assume~$Q_u$ is not a split neighbour of~$P$.
    Then~$Q_u$ is obtained from~$P$ by adding~$u$ to some $A\in P$.
    Let $R\notin N[P]$ be a common neighbour of~$Q_u$ and~$Q_v$.
    By Observation~\ref{obs:vopenneighbourhoodclique} we have $R\notin N_u(Q_u)$ and $R\notin N_v(Q_v)$ since $R\notin N[P]$.
    
    \textbf{Case 1:} Suppose~$Q_v$ is a split neighbour of~$P$.
    Then we have $\ord{P}<k$, so we know that $\ord{P(v)}\geq 4$ since~$P$ is a reconstruction candidate.
    Since $R\notin N_v(Q_v)$, we know that $\ord{R(v)}\leq 2$.
    If $\ord{R(v)}=2$, then~$R$ is obtained from~$Q_v$ by adding some~$v'$ to~$\{v\}$.
    In particular, then we have $Q_u(u)=Q_u(a)$ and $R(u)\neq R(a)$ for each $a\in A$.
    Since $R\notin N_u(Q_u)$, we know~$R$ is obtained from~$Q_u$ by moving every vertex in~$A$.
    In particular this means $\ord{A}=1$ and~$R$ is obtained from~$Q_u$ by moving the unique vertex $a\in A$.
    Observe that $a\notin P(v)$ since $\ord{P(v)}\geq 4$ so $P(v)\neq A$.
    Hence $\ord{R(v)}\geq 3$, a contradiction.
    So $\ord{R(v)}=1$, meaning~$R$ is the $v$-split neighbour of~$Q_u$ since $\ord{Q_u(v)}\geq 3$.
    In particular~$Q_u$ and~$Q_v$ have at most one common neighbour outside of~$N[P]$, contradicting that they are double-closed.
    
    \textbf{Case 2:}
    Suppose~$Q_v$ is not a split neighbour of~$P$.
    Then~$Q_v$ is obtained from~$P$ by adding~$v$ to some part $B\in P$.
    
    Our strategy for deriving a contradiction in this case is as follows.
    We will show that~$Q_u$ and~$Q_v$ have a unique common neighbour $R_1\notin N[P]$ obtained from~$Q_u$ by moving~$v$ and from~$Q_v$ by moving~$u$.
    We will then show that any other common neighbour $R_2\notin N[P]$ of~$Q_u$ and~$Q_v$ cannot be adjacent to~$R_1$.
    
    Suppose that $R_1\notin N[P]$ and $R_1\in N_v(Q_u)\cap N(Q_v)$.
    If $u\notin B$, define~$Q_u(B)$ to be the part of~$Q_u$ containing~$B$ (i.e.\ either~$B$ or $B\cup \{u\}$).
    If $u\in B$, define $Q_u(B)\coloneqq B\setminus\{u\}$.
    If $Q_u(B)=\emptyset$, we claim that~$R_1$ is obtained from~$Q_u$ by splitting~$v$.
    Otherwise we claim that~$R_1$ is obtained from~$Q_u$ by adding~$v$ to~$Q_u(B)$.
    
    Suppose that $B=\{u\}$, i.e.\ that $Q_u(B)=\emptyset$.
    Then we have $\{u,v\}\in Q_v$, and $\ord{A}\geq 3$ by statement~$2$ of Lemma~\ref{lem:reconstructioncandidatefacts}.
    This means that $\ord{Q_u(u)}=\ord{A\cup\{u\}}\geq 4$, so therefore we have $\ord{R_1(u)}\geq 3$, since $R_1\in N_v(Q_u)$ and $v\neq u$.
    Suppose~$R_1$ is not obtained from~$Q_u$ by splitting~$v$.
    Then $\ord{R_1(v)}\geq 2$ if $R_1(u)\neq R_1(v)$ and $\ord{R_1(v)}\geq 4$ if $R_1(u)= R_1(v)$.
    In either case, we have $\{u,v\}\notin Q_v$, a contradiction.
    
    Suppose that $B\neq \{u\}$.
    If $R_1(v)=R_1(b)$ for all $b\in B$, then~$R_1$ is obtained from~$Q_u$ by adding~$v$ to~$Q_u(B)$ as claimed.
    Suppose there is some $b\in B$ such that $R_1(v)\neq R_1(b)$.
    Then~$Q_v$ is obtained from~$R_1$ by adding~$b$ to~$R_1(v)$ since $B\cup\{v\}\in Q_v$ and $Q_v\notin N_v(R_1)$.
    If $R_1(v)\neq \{v\}$, then $R_1(v)\setminus\{v\}\in Q_u$ and $R_1(v)\setminus\{v\}$ is a non-empty subset of~$B$, which means $R_1(v)\setminus\{v\}=Q_u(B)$ as claimed.
    Suppose $R_1(v)=\{v\}$.
    Then $B=\{b\}$ so $R_1(b)\neq \{b\}$ by Observation~\ref{obs:identicalneighbours}.
    Since $b\neq u$ by hypothesis, and since~$R_1$ is obtained from~$Q_u$ by splitting~$v$, this means that $R_1(b)=\{u,b\}$, so we have $\{u\}\in Q_v$.
    This means that either $\{u,v\}\in P$, or $\{u\}\in P$, a contradiction to statement~$2$ of Lemma~\ref{lem:reconstructioncandidatefacts} in either case.
    
    In particular, there is at most one $R_1\notin N[P]$ such that $R_1\in N_v(Q_u)$ and $R_1\in N(Q_v)$.
    Moreover, any such~$R_1$ satisfies $R_1\in N_u(Q_v)$.
    Therefore, by symmetry between~$u$ and~$v$, there is at most one $R_1\notin N[P]$ satisfying both $R_1\in N_v(Q_u)$ and $R_1\in N_u(Q_v)$, and every other $R_2\notin N[P]$ satisfies $R_2\notin N_u(Q_v)$ and $R_2\notin N_v(Q_u)$.
    Observe that such an~$R_1$ exists if and only if $A\neq B$ or $uv\notin E(G)$, and let $R_2\notin N[P]$ be a common neighbour of~$Q_u$ and~$Q_v$ such that $R_2\notin N_u(Q_v)$ and $R_2\notin N_v(Q_u)$.
    
    Suppose $A=B$ and $uv\in E(G)$.
    Then $P(u)$,~$P(v)$, and~$A$ are all distinct, with $P(v), A\cup\{u\}\in Q_u$ and $A\cup\{v\}\in Q_v$.
    Since $R_2\notin N_u(Q_v)$ and $R_2\notin N_v(Q_u)$, moving from~$Q_u$ to~$R_2$ to~$Q_v$ must involve adding every $a\in A$ to~$Q_u(v)$ and removing all vertices other than~$v$ from~$Q_u(v)$.
    Therefore $\ord{A}+\ord{Q_u(v)\setminus\{v\}}\leq 2$, which means that $\ord{A}+\ord{P(v)}\leq 3$ since $P(v)=Q_u(v)$ because $P(u)\neq P(v)$.
    Hence, by statement~$2$ of Lemma~\ref{lem:reconstructioncandidatefacts}, we have $P(v)=\{v,v'\}$ and $A=\{a\}$ where $av'\in E(G)$.
    Therefore,~$a$ cannot be added to $\{v,v'\}$.
    Observing that $R_2\notin N_v(Q_u)$ and that the $v'$-split neighbour of~$Q_u$ is in~$N_v(Q_u)$,~$R_2$ is obtained from~$Q_u$ by adding~$v'$ to some part $C\in Q_u$.
    Then~$Q_v$ is obtained from~$R_2$ by adding~$a$ to~$\{v\}$.
    But we have $a\notin C$ since $av'\in E(G)$, so we have $C\cup \{v'\}\in Q_v$, a contradiction as $\{v'\}\in Q_v$ since we know $\{v,v'\}\in P$.
    
    Therefore, either $A\neq B$ or $uv\notin E(G)$.
    So in particular, there is exactly one vertex~$R_1$ such that $R_1\in N_v(Q_u)$ and $R_1\in N_u(Q_v)$.
    Since~$Q_u$ and~$Q_v$ are double-closed,~$R_2$ is adjacent to~$R_1$.
    Note that $R_2\notin N_v(R_1)$ by Observation~\ref{obs:vopenneighbourhoodclique} since we have both $Q_u\in N_v(R_1)$ and $R_2\notin N_v(Q_u)$.
    The situation is pictured in Figure~\ref{fig:doubleclosure}.
    
\begin{figure}[h!]
    \centering
    \begin{tikzpicture}[scale=0.8]
        \coordinate (A) at (0,0);
        \coordinate (B) at (30:3cm);
        \coordinate (C) at (-30:3cm);
        \coordinate (D) at (0:6cm);
        \coordinate (E) at (0:9cm);
    
        \node[circle, fill, inner sep=2pt, label=above:$P$] at (A) {};
        \node[circle, fill, inner sep=2pt, label=above:$Q_u$] at (B) {};
        \node[circle, fill, inner sep=2pt, label=above:$Q_v$] at (C) {};
        \node[circle, fill, inner sep=2pt, label=above:$R_1$] at (D) {};
        \node[circle, fill, inner sep=2pt, label=above:$R_2$] at (E) {};
        
        \draw[very thick] (A) -- (B) node[midway, above] {$u$};
        \draw[very thick] (A) -- (C) node[midway, above] {$v$};
        \draw[very thick] (D) -- (B) node[midway, above] {$v$};
        \draw[very thick] (D) -- (C) node[midway, above] {$u$};
        \draw[very thick] (D) -- (E) node[midway, above] {$\overline{v}$};
        \draw[very thick] (E) to[bend right=20] (B);
        \draw[very thick] (E) to[bend left=20] (C);
        \node at (5.8,1.7) {$\overline{u},\overline{v}$};
        \node at (5.8,-1.2) {$\overline{u},\overline{v}$};
    \end{tikzpicture}
    \caption{A situation in the proof of Lemma~\ref{lem:doubleclosureinformation}.
    An edge $ST$ with a label $x$ means $S$ and $T$ are obtained from each other by moving $x$.
    An edge $ST$ with a label $\overline{x}$ means $S$ and $T$ are \textit{not} obtained from each other by moving $x$.
    }
    \label{fig:doubleclosure}
\end{figure}
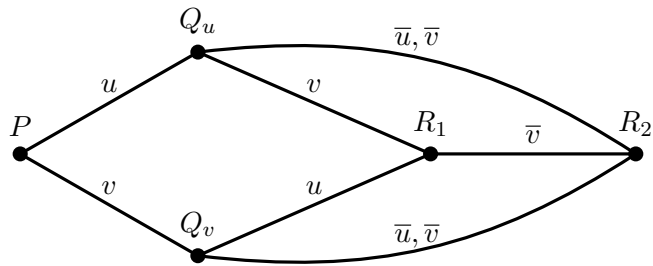

    Suppose there is some $b\in B$, $b\neq u$ such that $R_2(v)\neq R_2(b)$.
    Then~$R_1$ and~$Q_v$ are both obtained from~$R_2$ by adding~$b$ to~$R_2(v)$, meaning $R_1=Q_v$, a contradiction.
    So $R_2(v)=R_2(b)$ for all $b\in B$, $b\neq u$.
    But $Q_u(v)\neq Q_u(b)$ for all $b\in B$, $b\neq u$.
    Therefore if such a~$b$ exists, it is unique, and~$R_2$ is obtained from~$Q_u$ by adding~$b$ to~$Q_u(v)$.
    
    Suppose $B=\{b\}$ for some $b\neq u$.
    Then $A\cup\{u,b\}\in R_2$ and $\{v,b\}\in Q_v$.
    Since $R_2\in N_b(Q_u)$, we have $R_2\notin N_b(Q_v)$ by Observation~\ref{obs:vopenneighbourhoodclique} since~$Q_u$ and~$Q_v$ are not adjacent since they are double-closed.
    Therefore we have $A=\{v\}$, a contradiction to statement~$2$ of Lemma~\ref{lem:reconstructioncandidatefacts}.
    
    So $B=\{u\}$ or $B=\{b,u\}$ for some $b\in V(G)$.
    Observe that $\ord{R_1(v)}=\ord{B}$ in this case.
    We have $u,v\in B\cup \{v\}\in Q_v$ and $A\cup\{u\}\in Q_u$. 
    But~$Q_u$ can be obtained from~$Q_v$ in two steps via~$R_2$ without moving~$u$ or~$v$, so $A=P(v)$.
    
    Therefore $\ord{A}+\ord{B}\geq 4$ by statement~$2$ of Lemma~\ref{lem:reconstructioncandidatefacts}.
    If $B=\{u\}$, then we have $\ord{Q_u(v)}=\ord{A\cup\{u\}}\geq 4$, so we can calculate that $\ord{R_2(v)}\geq 3=\ord{R_1(v)}+2$.
    If $B=\{u,b\}$, then we have $\ord{Q_u(v)}\geq 3$, so we can calculate that $\ord{R_2(v)}=\ord{Q_u(v)\cup\{b\}}\geq 4=\ord{R_1(v)}+2$.
    However, since we know that $R_1\notin N_v(R_2)$, we know that~$R_1(v)$ and~$R_2(v)$ differ in size by at most~$1$.
    So we have obtained a contradiction in either case.
\end{proof}
    
Lemma~\ref{lem:doubleclosureinformation} will only be useful in the case where we have $k\geq \chi(G)+2$.
We prove the following lemma to deal with the case where $k=\chi(G)+1$.

\begin{lemma}\label{lem:lowerchiplusonecase}
    Suppose $k=\chi(G)+1$, and let $P\in V(\BlG)$ be a reconstruction candidate.
    Let~$uv$ be a non-edge of~$G$ and let $Q_u\in N_u(P)$ and $Q_v\in N_v(P)$.
    Then~$Q_u$ and~$Q_v$ have a common neighbour $R\notin N[P]$.
\end{lemma}

\begin{proof}
    We first construct~$R$.
    Suppose~$Q_u$ and~$Q_v$ are respectively the $u$-split and $v$-split neighbours of~$P$.
    Then let~$R$ be the partition obtained from~$Q_u$ by adding~$v$ to~$\{u\}$.
    Otherwise, without loss of generality,~$Q_u$ is obtained from~$P$ by adding~$u$ to a part $A\in P$, and~$Q_v$ is either obtained from~$P$ either by adding~$v$ to a part $B\in P$, or by splitting~$v$. (Observe that~$B$ need not be distinct from~$A$.)
    In the former case, let~$R$ be the partition obtained from~$P$ by adding~$u$ to~$A$ and by adding~$v$ to~$B$.
    In the latter case, let~$R$ be the partition obtained from~$P$ by adding~$u$ to~$A$ and by splitting~$v$.
    
    Observe that~$R$ is an independent set partition of~$G$ in every case, since $uv\notin E(G)$.
    Moreover, observe that if either~$Q_u$ or~$Q_v$ are split neighbours, then $\ord{P}<k$, and so we have $\ord{R}\leq \ord{P}+1\leq k$.
    On the other hand, if neither~$Q_u$ nor~$Q_v$ are split neighbours, then we have $\ord{R}\leq \ord{P}\leq k$.
    Therefore~$R$ is a vertex of~$\BlG$ in every case.
    Now suppose that $R\in N[P]$.
    Suppose~$Q_u$ and~$Q_v$ are both split neighbours.
    Then $R(u)=\{u,v\}$, and $\ord{P(u)},\ord{P(v)}\geq 2$. Moreover, we have $\ord{P}<\ord{Q_u}\leq k$, so by statement~$3$ of Lemma~\ref{lem:reconstructioncandidatefacts}, we have $\ord{P(u)}\geq 4$.
    So~$R$ is obtained from~$P$ by adding~$u$ to~$\{v\}$, which is not possible as $\{v\}\notin P$.
    Therefore, without loss of generality,~$Q_u$ is obtained from~$P$ by adding~$u$ to some part $A\in P$.
    Then, we either have $A\cup\{u\}\in R$, or $A\cup\{u,v\}\in R$, or $A\cup \{u\}\setminus\{v\}\in R$.
    Suppose $A=\{a\}$ for some $a\in V(G)$.
    Then~$R$ is obtained from~$P$ by adding~$a$ either to~$\{u\}$ or to $\{u,v\}$, a contradiction to statement~$2$ of Lemma~\ref{lem:reconstructioncandidatefacts} in either case.
    So $\ord{A}\geq 2$, which means that~$R$ is obtained from~$P$ by adding~$u$ to~$A$, which means that $R=Q_u$, a contradiction.
    So we have $R\notin N[P]$.
\end{proof}

We now complete the proof of Theorem~\ref{thm:mainlower}.

\begin{proof}[Proof of Theorem~\ref{thm:mainlower}]
    Let $n\coloneq \ord{V(G)}$.
    Observe that we can determine~$n$ from~$\BlG$ by considering any vertex $P\in V(\BlG)$ with a maximal number of components in its neighbourhood~$\mathcal{N}(P)$, since then~$P$ is a reconstruction candidate by definition, so~$\mathcal{N}(P)$ has precisely~$n$ components by Lemma~\ref{lem:reconstructioncandidatefacts}.

    By Lemma~\ref{lem:chihajnal},~$G$ has an independent set partition~$P^\dagger$ into~$\chi(G)$ parts, where each part has size at least~$4$.
    Observe that $P^{\dagger}$ is a reconstruction candidate by Lemma~\ref{lem:reconstructioncandidatefacts}.
    
    Suppose $k=\chi(G)+1$, and let $P\in V(\BlG)$ be a reconstruction candidate.
    Let $Q,Q'\in N(P)$ be neighbours of $P$ which are not adjacent to each other.
    Then since we have $\ord{P}\geq \chi(G)=k-1$, we know by Lemma~\ref{lem:doubleclosureinformation} that~$Q$ and~$Q'$ are not double-closed. 
    On the other hand, if $k>\chi(G)+1$, then for any part $A\in P^\dagger$ and any two vertices $u,v\in A$, the $u$-split neighbour of~$P^\dagger$ and the $v$-split neighbour of~$P^\dagger$ are double-closed by Lemma~\ref{lem:splitclosureinformation}.
    Therefore, we know that $k>\chi(G)+1$ if and only if~$\BlG$ contains a reconstruction candidate~$P$ with two neighbours~$Q$ and~$Q'$ which are double-closed.
    In other words, we can determine whether $k>\chi(G)+1$ or $k=\chi(G)+1$.
    
    Now, for each reconstruction candidate $P\in V(\BlG)$, we construct a candidate graph $G_P$ on vertex set $\{1,\ldots,n\}$ as follows.
    Let~$\zeta$ be an arbitrary bijection from $\{1,\ldots,n\}$ to the components of~$\mathcal{N}(P)$.
    If $k=\chi(G)+1$, add an edge~$uv$ if and only if there exist some $Q_u\in N_{\zeta(u)}(P)$ and $Q_v\in N_{\zeta(v)}(P)$ such that every common neighbour of~$Q_u$ and~$Q_v$ belongs to~$N[P]$.
    Then, by Lemmas~\ref{lem:splitclosureinformation} and~\ref{lem:lowerchiplusonecase},~$G_P$ is isomorphic to a subgraph of~$G$, and~$G_{P^\dagger}$ is isomorphic to~$G$.
    In this case, we can recognise~$G$ by selecting any largest graph among the set of candidate graphs.
    If $k>\chi(G)+1$, add an edge~$uv$ if and only if for every $Q_u\in N_{\zeta(u)}(P)$ and $Q_v\in N_{\zeta(v)}(P)$, the vertices~$Q_u$ and~$Q_v$ are not double-closed.
    Then, by Lemma~\ref{lem:doubleclosureinformation}, we know that~$G_P$ is isomorphic to a supergraph of~$G$, and~$G_{P^\dagger}$ is isomorphic to~$G$.
    In this case, we can recognise~$G$ by selecting any smallest graph among the set of candidate graphs.
\end{proof}

\paragraph{Acknowledgment}\mbox{}\\*
The author thanks Jan van den Heuvel for the motivating discussions and guidance throughout the writing of this paper.

%% file: A_uppertheorem.tex
\newpage
\appendix
\section{Strengthening Theorem~\ref{thm:mainupperhalfcomplete}}\label{section:mainupper}

In this appendix, we will strengthen Theorem~\ref{thm:mainupperhalfcomplete} by giving a complete classification of the pairs of tuples~$(G_1,k_1)$ and~$(G_2,k_2)$ such that the upper Bell $k_1$-colouring graph of~$G_1$ is isomorphic to the upper Bell $k_2$-colouring graph of~$G_2$.
Let~$G$ be a graph on~$n$ vertices.
Recall that~$G'$ is the graph obtained from~$G$ by removing all of its universal vertices.
Define~$\GClaw$ to be the graph obtained from~$G'$ as follows: for each independent set $\{v_1,v_2,v_3\}\subseteq V(G)$ such that $v_1$,~$v_2$, and~$v_3$ are adjacent to every other vertex of~$G$, add all edges between $v_1$,~$v_2$, and~$v_3$ and add a vertex to~$G'$ adjacent to every vertex other than $v_1$,~$v_2$, and~$v_3$.
Equivalently, by noting that a vertex is universal in~$G$ if and only if it is isolated in its complement~$\comp{G}$, we can write $\comp{\GClaw}=\comp{G}^{\lowddagger}$, where~$\comp{G}^{\lowddagger}$ is the graph obtained from~$\comp{G}$ by removing all of its isolated vertices and replacing all triangle components with claw components.
As usual, for each natural number~$s$, denote by~$K_s$ and~$P_s$ the clique and path on~$s$ vertices respectively, and use $H_1+H_2$ to denote the disjoint union of two graphs~$H_1$ and~$H_2$.

\begin{theorem}\label{thm:mainuppercomplete}
    For each~$i\in \{1,2\}$, let~$G_i$ be a graph on~$n_i$ vertices and let~$k_i$ be a natural number.
    Then $\mathcal{B}_{\geq {k_1}}(G_1)\cong\mathcal{B}_{\geq {k_2}}(G_2)$ if and only if at least one of the following conditions holds for both $i\in \{1,2\}$.
    \begin{enumerate}
        \item\label{case:upperone}~$k_i>n_i$.
        \item\label{case:uppertwo}~$k_i=n_i$, or $k_i\leq n_i$ and~$G_i$ is a clique.
        \item\label{case:upperthree}~$G_1^{\claw} \cong G_2^{\claw}$ and $k_i = n_i-1$.
        \item\label{case:upperfour}~$G_1'\cong G_2'$, and $\chi(G_i)+1\leq k_i \leq n_i-2$, and $n_1-k_1=n_2-k_2$.
        \item\label{case:upperfive}~$G_1'\cong G_2'$ and $k_i\leq \chi(G_i)$.
        \item\label{case:uppersix}~$k_i\leq n_i-1$ and ${G_i'}\cong K_{N_i-1}+K_1$, where $N_i\coloneqq\ord{\mathcal{B}_{\geq {k_i}}(G_i)}$.
        \item\label{case:upperseven}~$k_i\leq n_i-1$ and ${G_i'}\cong K_{3}+K_1$, or $k_i=n_i-1$ and ${G_i'}\cong \overline{K_3}$.
        \item\label{case:uppereight}~$k_i\leq n_i-2$ and ${G_i'}\cong \overline{K_3}$, or $k_i=n_i-1$ and ${G_i'}\cong P_3+K_1$.
    \end{enumerate}
\end{theorem}

Given any of the eight conditions on $G_1$, $G_2$,~$k_1$, and~$k_2$ in Theorem~\ref{thm:mainuppercomplete}, it will not be difficult to show that~$\mathcal{B}_{\geq k_1}(G_1)$ and~$\mathcal{B}_{\geq k_2}(G_2)$ are isomorphic.
On the other hand, given an upper Bell colouring graph~$\BuG$ of a graph~$G$ on~$n$ vertices, we will show that in all but a small number of cases we can distinguish between the cases where $k\leq n-2$, $k=n-1$, and $k=n$.
(Observe that the case where $k>n$ is easy, since $k>n$ if and only if~$\BuG$ contains no vertices.)
The tools developed in Section~\ref{section:main} will then allow us to handle the case where $k\leq n-2$.
Moreover, we will show that in the case where $k=n-1$, we can recognise a partition~$P$ from which we can determine~$\GClaw$ using the local structure of~$\mathcal{B}_{\geq k}(G)$ near~$P$.
Handling the small number of remaining cases will complete the proof of Theorem~\ref{thm:mainuppercomplete}.

Throughout this appendix, we will use the same terminology introduced at the start of Section~\ref{section:main}.
Given a graph~$G$ on~$n$ vertices, we will use~$P^*_G$ to represent the partition of~$V(G)$ into~$n$ parts, each of size~$1$, and omit the subscript when it is clear from context.

\begin{lemma}\label{lem:uppermainuseful2}
    Let~$G_1$ and~$G_2$ be graphs on~$n_1$ and~$n_2$ vertices respectively.
    If $\GClaw_1\cong \GClaw_2$, then $\mathcal{B}_{\geq {n_1-1}}(G_1)\cong\mathcal{B}_{\geq {n_2-1}}(G_2)$.
\end{lemma}

\begin{proof}
    For any graph~$H$, let~$H^+$ denote the graph obtained from~$H$ by adding a universal vertex.
    Observe that, for any graph~$G$ on~$n$ vertices,~$\mathcal{B}_{\geq n-1}(G)$ is just the closed neighbourhood~$\mathcal{N}[P^*]$, which is isomorphic to~$L(\comp{G})^+$ by Lemma~\ref{lem:linegraph2}.
    
    Let $i\in \{1,2\}$ be fixed.
    Let~$n_i'$ denote the number of vertices of~$G_i'$ and let~$n_i^{\lowclaw}$ denote the number of vertices of~$\GClaw$.
    Then we can calculate
    \begin{align}\label{eqn:appendixn-1}
        \mathcal{B}_{\geq n_i^{\lowclaw}-1}(\GClaw_i)\cong L(\comp{\GClaw})^+\cong L(\comp{G'})^+\cong \mathcal{B}_{\geq n_i'-1}(G'_i),
    \end{align}
    where we have used the fact that $L(\comp{\GClaw})\cong L(\comp{G'})$ since $L(K_{1,3})\cong L(K_3)$.
    
    Since $\GClaw_1\cong \GClaw_2$ and $n_1^{\lowclaw}=n_2^{\lowclaw}$, we know that $\mathcal{B}_{\geq n_1^{\lowclaw}-1}(\GClaw_1)\cong \mathcal{B}_{\geq n_2^{\lowclaw}-1}(\GClaw_2)$.
    Therefore, by~\eqref{eqn:appendixn-1}, we know that $\mathcal{B}_{\geq n_1'-1}(G_1')\cong \mathcal{B}_{\geq n_2'-1}(G_2')$.
    The result follows by Observation~\ref{obs:addfulldegree}.
\end{proof}

We now prove one direction of Theorem~\ref{thm:mainuppercomplete}.

\begin{proof}[Proof of Theorem~\ref{thm:mainuppercomplete}, `if' direction]
    We show that $\mathcal{B}_{\geq k_1}(G_1)\cong\mathcal{B}_{\geq k_2}(G_2)$ in each of the eight cases.
    Let $i\in \{1,2\}$ be fixed.
    \begin{enumerate}
        \item Suppose that $k_i>n_i$.
        Then~$\mathcal{B}_{\geq k_i}(G)$ is the empty graph with no vertices.
        \item Suppose that either $k_i=n_i$, or~$G_i$ is a clique and $k_i\leq n_i$.
        Then~$\mathcal{B}_{\geq k_i}(G)$ is a graph on a single vertex.
        \item Suppose that $\GClaw_1 \cong \GClaw_2$ and $k_i = n_i-1$.
        Then the claim follows by Lemma~\ref{lem:uppermainuseful2}.
        \item Suppose that $G_1'\cong G_2'$, and that $\chi(G_i)+1\leq k_i \leq n_i-2$, and $n_1-k_1=n_2-k_2$.
        Then the claim follows by Lemma~\ref{lem:uppermainuseful}.
        \item Suppose that $G_1'\cong G_2'$, and that $k_i\leq \chi(G_i)$. Then for each $i\in \{1,2\}$ we have $\mathcal{B}_{\geq k_i}(G_i)\cong \mathcal{B}(G_i)$, and so the claim follows by Theorem~\ref{thm:maincomplete}.
        \item Suppose that $k_i\leq n_i-1$, and that $G_i'\cong K_{N}+K_1$ for some $N\geq 1$.
        Let~$v$ be the unique vertex in the~$K_1$ component of~$G_i'$.
        Then there are exactly~$N$ partitions $P\in V(\mathcal{B})$, namely the partition~$P^*$ into $N+1$ parts of size~$1$, and the~$N$ partitions where $\ord{P(v)}=2$ and every other part has size~$1$.
        Moreover, each of these partitions are adjacent to each other.
        Hence~$\mathcal{B}_{\geq k_i}(G_i)$ is a clique on $N+1$ vertices.
        \item Suppose that either $k_i\leq n_i-1$ and $G_i'\cong K_3+K_1$, or $k_i=n_i-1$ and $G_i'\cong \comp{K_3}$.
        Then in both cases, there are exactly four partitions of~$G_i'$ into at least~$k_i$ parts, and these partitions are all adjacent to each other in~$\mathcal{B}_{\geq k_i}(G_i)$.
        Hence~$\mathcal{B}_{\geq k_i}(G_i)$ is isomorphic to the clique~$K_4$.
        \item Suppose either $k_i\leq n_i-2$ and $G_i'\cong\comp{K_3}$, or $k_i=n_i-1$ and $G_i'\cong P_3+K_1$.
        Then in both cases, there are exactly five partitions of~$G_i'$ into at least~$k_i$ parts, and with one exception these partitions are all adjacent to each other in~$\mathcal{B}_{\geq k_i}(G_i)$.
        Hence~$\mathcal{B}_{\geq k_i}(G_i)$ is isomorphic to~$K_5^-$, defined as the graph formed by removing any edge from the clique~$K_5$.\qedhere
    \end{enumerate}
\end{proof}

It remains to prove the other direction of Theorem~\ref{thm:mainuppercomplete}.
It will be useful to explicitly write down the recovery algorithm we used in the proofs of Theorem~\ref{thm:mainupperhalfcomplete} and Theorem~\ref{thm:main}.
Let~$G$ be a graph on~$n$ vertices, let $k\leq n-1$ be a natural number, and let $\mathcal{B}\coloneqq \mathcal{B}_{\geq k}(G)$.
We build a function $\phi_{\mathcal{B}}:V(\mathcal{B})\rightarrow \mathcal{G}$, where~$\mathcal{G}$ denotes the set of all graphs.
Let $P\in V(\mathcal{B})$ be given.
If~$\mathcal{N}(P)$ is not isomorphic to the line graph of any graph, set~$\phi_{\mathcal{B}}(P)$ equal to any graph.
On the other hand, suppose~$\mathcal{N}(P)$ is isomorphic to the line graph~$L(H)$ of some graph~$H$.
Let~$H^{\ddagger}$ be the graph obtained from~$H$ by removing all of its isolated vertices, and by replacing all of its triangle components with claw components.
Let~$T^{\ddagger}$ denote the number of triangles in~$H^{\ddagger}$, and let~$T_P$ denote the number of triangles~$\Delta$ in~$\mathcal{N}(P)$ such that the three vertices of~$\Delta$ have a common neighbour outside of~$N[P]$.
If $T_P-T^{\ddagger}>0$ and~$H^{\ddagger}$ has at least $T_P-T^{\ddagger}$ claw components, set~$\phi_\mathcal{B}(P)$ to be the complement of the graph obtained from~$H^{\ddagger}$ by replacing~$T_P-T^{\ddagger}$ claw components with triangle components.
Otherwise, set~$\phi_\mathcal{B}(P)$ equal to the complement of~$H^{\ddagger}$.

We will study the behaviour of the functions~$\phi_{\mathcal{B}_{\geq k}(G)}$ for various upper Bell colouring graphs~$\mathcal{B}_{\geq k}(G)$.
We first record the following, whose proof follows immediately from the first two paragraphs of the proof of Lemma~\ref{lem:mainreconstruction} in Subsection~\ref{subsect:mainstep2}.

\begin{lemma}\label{lem:explicitmainreconstruction}
    Let~$G$ be a graph on~$n$ vertices, and let $k\leq n-2$.
    Then the set $\Omega_5\subseteq V(\mathcal{B}_{\geq k}(G))$ of $P^*$-candidates is non-empty, and we have $\phi_{\BuG}(P)\cong G'$ for every $P\in \Omega_5$.
\end{lemma}

\begin{lemma}\label{lem:P^*reconstructionn-1}
Let~$G$ be a graph on~$n$ vertices.
Then $\phi_{\mathcal{B}_{\geq n-1}(G)}(P^*)\cong \GClaw$.
\end{lemma}

\begin{proof}
    Recall that we can write $\comp{\GClaw}=\comp{G}^{\lowddagger}$, where~$\comp{G}^{\lowddagger}$ is the graph obtained from~$\comp{G}$ by removing all isolated vertices and replacing all triangle components with claw components.
    
    By Lemma~\ref{lem:linegraph2}, we have $\mathcal{N}(P^*)\cong L(\comp{G})$.
    Since $P^*\in V(\mathcal{B}_{\geq n-1}(G))$ is a universal vertex, there are no vertices $R\notin N[P^*]$.
    So we know that $T_{P^*}=0$.
    Therefore,~$\phi_{\mathcal{B}_{\geq n-1}(G)}(P^*)$ is the complement of~$\comp{G}^{\lowddagger}$ by definition, as claimed.
\end{proof}

As discussed at the start of this appendix, we would like to show that, given an upper Bell $k$-colouring graph~$\BuG$ of a graph~$G$ on~$n$ vertices, in all but a small number of cases we can determine whether $k\leq n-2$, or $k=n-1$, or $k=n$.
The following lemma shows that the existence of a universal vertex in~$\BuG$ almost always guarantees that $n-1\leq k\leq n$.
Recall that~$K_5^-$ is the graph formed by removing any edge from the clique~$K_5$.

\begin{lemma}\label{lem:uppermaindistinguish}
    Let~$G$ be a graph on~$n$ vertices and let $k\leq n$ be a natural number.
    Suppose~$\BuG$ is not a clique or isomorphic to~$K_5^-$.
    Then $k\leq n-2$ if and only if~$\BuG$ contains no universal vertices.
\end{lemma}

\begin{proof}
    If $n-1\leq k\leq n$, then~$P^*$ is a universal vertex.
    Now suppose $k\leq n-2$.
    We consider two cases.
    
    \textbf{Case 1:} Suppose there exist two edges $\{u,v\}$ and $\{w,x\}$ of~$\comp{G}$ which are not incident with each other.
    Let~$Q_1$ be the partition obtained from~$P^*$ by merging~$\{u\}$ and~$\{v\}$.
    Let~$Q_2$ be the partition obtained from~$P^*$ by merging~$\{w\}$ and~$\{x\}$.
    Let~$R$ be the partition obtained from~$P^*$ by replacing $\{u\}$, $\{v\}$,~$\{w\}$, and~$\{x\}$ by $\{u,v\}$ and $\{w,x\}$.
    Observe that~$R$ is not a neighbour of~$P^*$, and that~$Q_1$ is not a neighbour of~$Q_2$.
    Moreover,~$Q_1$ and~$Q_2$ are the only common neighbours of~$P^*$ and~$R$.
    Hence every vertex in~$\BuG$ has at least one non-neighbour, i.e.~$\BuG$ contains no universal vertices.
        
    \textbf{Case 2:} Suppose all edges of~$\comp{G}$ are incident with each other.
    If~$\comp{G}$ contains a triangle, then every edge of~$\comp{G}$ lies in this triangle, so~$\BuG$ is isomorphic to~$K_5^-$ by direct computation, a contradiction.
    So~$\comp{G}$ does not contain a triangle, which means that every edge of~$\comp{G}$ shares a common vertex~$v$.
    Therefore, the only vertices of~$\mathcal{B}_{\geq k}(G)$ are~$P^*$ and the~$\ord{E(\comp{G})}$ partitions obtained from~$P^*$ by merging~$\{u\}$ and~$\{v\}$ for some $uv\in E(\comp{G})$.
    In other words, we can write~$\mathcal{B}_{\geq k}(G)= N[P^*]$. However, by Lemma~\ref{lem:linegraph2}, we know that~$N(P^*)$ is a clique, so therefore $N[P^*]=\mathcal{B}_{\geq k}(G)$ is also a clique, a contradiction.
\end{proof}

We now turn our attention towards the special case where $k=n-1$ for an upper Bell $k$-colouring graph~$\BuG$ of a graph~$G$ on~$n$ vertices.

\begin{lemma}\label{lem:uppern-1case}
    Let~$G_1$ and~$G_2$ be graphs on~$n_1$ and~$n_2$ vertices respectively and let $k_1=n_1-1$ and $k_2=n_2-1$.
    Suppose $\mathcal{B}_{\geq k_1}(G_1)\cong\mathcal{B}_{\geq k_2}(G_2)$.
    Then $\GClaw_1\cong \GClaw_2$.
\end{lemma}

\begin{proof}
    Let $\mathcal{B}_1\coloneqq\mathcal{B}_{\geq k_1}(G_1)$ and $\mathcal{B}_2\coloneqq\mathcal{B}_{\geq k_2}(G_2)$.
    Observe that, for each $i\in \{1,2\}$, we know that~$P^*_{G_i}$ is a universal vertex in~$\mathcal{B}_i$ since $k_i=n_i-1$.
    By hypothesis, there is a graph isomorphism $\theta:V(\mathcal{B}_1)\rightarrow V(\mathcal{B}_2)$.
    Therefore,~$\theta(P^*_{G_1})$ is a universal vertex in~$\mathcal{B}_2$.
    Hence, the map $\xi:V(\mathcal{B}_2)\rightarrow V(\mathcal{B}_2)$ which swaps~$\theta(P_{G_1})$ with~$P^*_{G_2}$ and sends every other vertex to itself is a graph isomorphism.
    It follows that $\phi_{\mathcal{B}_2}(P^*_{G_2})\cong \phi_{\mathcal{B}_1}(P^*_{G_1})$, since $\xi\circ \theta$ is a graph isomorphism from~$V(\mathcal{B}_1)$ to~$V(\mathcal{B}_2)$ which carries~$P^*_{G_1}$ to~$P^*_{G_2}$.
    
    By Lemma~\ref{lem:P^*reconstructionn-1}, we know that $\phi_{\mathcal{B}_i}(P_{G_i}^*)\cong \GClaw_i$ for both $i\in \{1,2\}$.
    The result follows.
\end{proof}

We now complete the proof of Theorem~\ref{thm:mainuppercomplete} by applying the results we have proved in this section, and tidying up the remaining cases.

\begin{proof}[Proof of Theorem~\ref{thm:mainuppercomplete}, `only if' direction]
    Suppose $\mathcal{B}_{\geq {k_1}}(G_1)\cong\mathcal{B}_{\geq {k_2}}(G_2)$.
    Let $i\in \{1,2\}$ be fixed.
    As shorthand, define $\mathcal{B}_i\coloneqq\mathcal{B}_{\geq {k_i}}(G_i)$.
     
    We show that in each of the following cases, at least one of the eight conditions in the statement of Theorem~\ref{thm:mainuppercomplete} holds.
        
    \textbf{Case 1:} Suppose~$\mathcal{B}_i$ has no vertices.
    Then we know that $k_i>n_i$, so condition~$\ref{case:upperone}$ holds.

    \textbf{Case 2:} Suppose~$\mathcal{B}_i$ has exactly one vertex.
    Then we know that either $k_i=n_i$, or~$G_i$ is a clique and $k_i\leq n_i$, so condition~$\ref{case:uppertwo}$ holds.

    \textbf{Case 3:} Suppose~$\mathcal{B}_i$ has at least two vertices, and is neither a clique nor isomorphic to~$K_5^-$.
    Then we know that $k_i\leq n_i-1$, since otherwise~$\mathcal{B}_i$ would only have one vertex.
    \begin{itemize}
        \item \textbf{Subcase 3.1:} Suppose~$\mathcal{B}_i$ contains a universal vertex.
        Then by Lemma~\ref{lem:uppermaindistinguish} we know that $k_i=n_i-1$.
        Hence, by Lemma~\ref{lem:uppern-1case}, we know that $\GClaw_1\cong \GClaw_2$.
        So condition~$\ref{case:upperthree}$ holds.

        \item \textbf{Subcase 3.2:} Suppose~$\mathcal{B}_i$ contains no universal vertices.
        Then by Lemma~\ref{lem:uppermaindistinguish} we know that $k_i\leq n_i-2$.
        By Theorem~\ref{thm:mainupperhalfcomplete}, we know that $G_1'\cong G_2'$, and that either $k_i\leq \chi(G_i)$ holds, or both $\chi(G_i) + 1 \leq  k_i\leq n_i-2$ and $n_1-k_1=n_2-k_2$ hold.
        In the latter case, condition~$\ref{case:upperfour}$ holds.
        In the former case, condition~$\ref{case:upperfive}$ holds.
    \end{itemize}
        
    \textbf{Case 4:} Suppose~$\mathcal{B}_i$ is a clique on $N_i\geq 2$ vertices.
    Then $k_i\leq n_i-1$, or else~$\mathcal{B}_i$ would only have one vertex.
    Moreover, by direct computation,~$\phi_{\mathcal{B}_i}(P)$ is isomorphic to the disjoint union of cliques $K_{N_i-1}+K_1$ for each $P\in V(\mathcal{B}_i)$.
    
    If $k_i=n_i-1$, then we know that $\phi_{\mathcal{B}_i}(P^*_{G_i})\cong \GClaw_i$ by Lemma~\ref{lem:P^*reconstructionn-1}, so in particular we know that $\GClaw_i\cong K_{N_i-1}+K_1$.
    If $k_i\leq n_i-2$, then by Lemma~\ref{lem:explicitmainreconstruction} we know that $\Omega_5\subseteq V(\mathcal{B}_i)$ is a non-empty set of vertices, and that $\phi_{\mathcal{B}_i}(P)\cong G'_i$ for every $P\in \Omega_5$.
    So in particular we know that $G'_i\cong K_{N_i-1}+K_1$.
    Observe however that $K_{N_i-1}+K_1$ does not contain any independent sets of size~$3$, so in this case we also have $\GClaw_i \cong K_{N_i-1}+K_1$ by definition of~$\GClaw_i$.
    \begin{itemize}
        \item \textbf{Subcase 4.1:}
        If $N_i\neq 4$, then~$\comp{\GClaw_i}$ does not contain any~$K_{1,3}$ components, so $\GClaw_i\cong G'_i$ by definition of~$\GClaw_i$.
        Hence~$G_i'$ is isomorphic to $K_{N_i-1}+K_1$.
        So condition~$\ref{case:uppersix}$ holds.
        
        \item \textbf{Subcase 4.2:}
        If $N_i=4$, then~$\comp{\GClaw_i}$ is a claw.
        Therefore, by definition of~$\GClaw_i$, we know that~$\comp{G_i'}$ is either a claw or a triangle.
        In other words,~$G_i'$ is isomorphic either to $K_3+K_1$, or to~$\comp{K_3}$.
        If $G_i'\cong \comp{K_3}$, then by direct computation we know that~$\mathcal{B}_i$ is not a clique if $k_i\leq n_i-2$.
        Therefore $k_i=n_i-1$ in this case.
        So condition~$\ref{case:upperseven}$ holds.
    \end{itemize}
    
    \textbf{Case 5:} Suppose~$\mathcal{B}_i$ is isomorphic to~$K_5^-$.
    Then we have $k_i\leq n_i-1$, or else~$\mathcal{B}_i$ would only have one vertex.
    Moreover, for each $P\in V(\mathcal{B}_i)$, by direct computation we know that~$\phi_{\mathcal{B}_i}(P)$ is isomorphic either to~$\comp{K_3}$ or $P_3+K_1$.
    
    If $k_i\leq n_i-2$, then by Lemma~\ref{lem:explicitmainreconstruction} we know that $\Omega_5\subseteq V(\mathcal{B}_i)$ is a non-empty set of vertices, and that $\phi_{\mathcal{B}_i}(P)\cong G'_i$ for every $P\in \Omega_5$.
    So therefore~$G'_i$ is isomorphic either to~$\comp{K_3}$, or to $P_3+K_1$.
    If $k_i=n_i-1$, then we know that $\phi_{\mathcal{B}_i}(P^*_{G_i})\cong \GClaw_i$ by Lemma~\ref{lem:P^*reconstructionn-1}, so~$\GClaw_i$ is isomorphic either to~$\comp{K_3}$, or to $P_3+K_1$.
    Observe however that neither~$K_3$ nor $\comp{P_3+K_1}$ contain a claw as a component.
    Therefore by definition of~$\GClaw_i$, we conclude that~$G'_i$ is isomorphic either to~$\comp{K_3}$ or to $P_3+K_1$ in this case also.

    If~$G'_i$ is isomorphic to~$\comp{K_3}$, then direct computation shows that $\mathcal{B}_{\geq n_i-1}(G_i)\cong K_4\ncong K_5^-$, meaning that $k_i\leq n_i-2$ in this case.
    If~$G'_i$ is isomorphic to $P_3+K_1$, then direct computation shows that if $k_i\leq n_i-2$, then~$\mathcal{B}_i$ has at least~$6$ vertices, and so in particular is not isomorphic to~$K_5^-$.
    Therefore we know that $k_i=n_i-1$ in this case.
    So condition~$\ref{case:uppereight}$ holds.
\end{proof}